\documentclass[12pt,reqno]{amsart}

\newtheorem{lemma}{Lemma}[section]
\newtheorem{definition}{Def\mbox inition}[section]
\newtheorem{proposition}{Proposition}[section]
\newtheorem{theorem}{Theorem}[section]
\newtheorem{remark}{Remark}[section]
\newtheorem{exam}{Example}[section]

\newtheorem{corollary}{Corollary}[section]

\usepackage{amsmath,amssymb,amsfonts}
\usepackage{amsthm}
\usepackage{mathrsfs}
\usepackage{amssymb,amsmath,amsthm}
\usepackage{epsfig,amssymb,amsmath,amsthm,graphicx,psfrag}
\usepackage[all]{xy}
\usepackage{enumerate}
\usepackage{color}
\usepackage{tikz}
\usepackage{array}
\usepackage{multirow}
\usepackage{multicol}
\usepackage{lipsum}
\usetikzlibrary{shapes.geometric}
\usetikzlibrary{calc}

\theoremstyle{plain}

\newsavebox{\myboxb}

\newcommand{\mypictureb}[1][]{
\begin{tikzpicture}[#1]

\def\dist{2.5}
\node at ($(-5.5,1.5)+(15:2.5)$) {$Q:$};

    \draw[fill=black] ($(0,0)+(90:2)$) circle (.08);

    \draw[fill=black] ($(0,0)+(0:2)$) circle (.08);

    \draw[fill=black] ($(0,0)+(45:2)$) circle (.08);
    
    \draw[fill=black] ($(0,0)+(270:2)$) circle (.08);

    \draw[fill=black] ($(0,0)+(225:2)$) circle (.08);
    
    \draw[fill=black] ($(0,0)+(135:2)$) circle (.08);

\draw[->,shorten <=7pt, shorten >=7pt] ($(-.3,4)+(280:2)$) arc (85:140:1.7);
\node at ($(-.4,.2)+(105:2)$) {$\alpha_0$};

\draw[->,shorten <=7pt, shorten >=7pt] ($(0,3)+(228:2.1)$) arc (135:180:2);
\node at ($(-.8,2.5)+(228.5:2.1)$) {$\alpha_1$};

\draw[->,shorten <=7pt, shorten >=7pt] ($(.5,.4)+(300:2)$) arc (325:355:2.8);
\node at ($(.1,1.7)+(338.5:2.5)$) {$\alpha_{k-1}$};
  
\draw[->,shorten <=7pt, shorten >=7pt] ($(-.35,0)+(280:2)$) arc (270:310:2.1);
\node at ($(.1,4.3)+(315-22.5:2.5)$) {$\alpha_k$};

\draw[->,shorten <=7pt] ($(1,-.6)+(78:2)$) arc (43:80:2.2);
\draw[->,shorten >=7pt] ($(.3,1.3)+(328:2)$) arc (5:50:1.7);
\draw[->,shorten <=7pt] ($(-1.9,2)+(270:2)$) arc (170:210:1.7);
\draw[->,shorten >=7pt] ($(.5,-.5)+(215:2)$) arc (240:270:2.3);
\node at ($(1.7,.2)+(270-20:2.5)$) {$\alpha_i$};
\node at ($(.15,-1.9)+(25:2.5)$) {$\alpha_{k-2}$};
\node at ($(-4.6,-1.5)+(20:2.5)$) {$\alpha_{i-2}\ $};
\node at ($(-1.9,.3)+(315-25:2.5)$) {$\ \alpha_{i-1}$};

\foreach \ang in {310,315,320,176,179.5,183}{
  \draw[fill=black] ($(0,0)+(\ang:2)$) circle (.02);
}
\end{tikzpicture}
}

\makeatletter
\g@addto@macro{\endabstract}{\@setabstract}
\newcommand{\authorfootnotes}{\renewcommand\thefootnote{\@fnsymbol\c@footnote}}
\makeatother
\usepackage{hyperref}
\hypersetup{pdfborder=0 0 0}

\makeatletter
\tikzset{my loop/.style =  {to path={
  \pgfextra{}
  [looseness=12,min distance=10mm]
  \tikz@to@curve@path},font=\sffamily\small
  }}  
\makeatletter


\begin{document}
\title{About UMP algebras and a special classif\mbox ication case} 

\author[Caranguay-Mainguez]{Jhony Caranguay-Mainguez}
\address{Universidad de Antioquia, Instituto de Matem\'aticas}
\curraddr{Calle 67 No. 53-108, Medell\'in, Colombia}
\email{jhony.caranguay@udea.edu.co}

\author[Franco]{Andr\'es Franco}
\address{Universidad de Antioquia, Instituto de Matem\'aticas}
\curraddr{Calle 67 No. 53-108, Medell\'in, Colombia}
\email{andres.francol@udea.edu.co}

\author[Reynoso-Mercado]{David Reynoso-Mercado}
\address{Universidad de Antioquia, Instituto de Matem\'aticas}
\curraddr{Calle 67 No. 53-108, Medell\'in, Colombia}
\email{david.reynoso@udea.edu.co}

\author[Rizzo]{Pedro Rizzo}
\address{Universidad de Antioquia, Instituto de Matem\'aticas}
\curraddr{Calle 67 No. 53-108, Medell\'in, Colombia}
\email{pedro.hernandez@udea.edu.co}

\maketitle

\begin{abstract}
The class of UMP algebras arises in several classif\mbox ication problems in the context of derived categories of f\mbox inite-dimensional algebras. In this paper we def\mbox ine the class of UMP algebras and develop algebraic combinatorics tools in order to present a characterization of this class of algebras which are \textit{locally monomial} (see Def\mbox inition \ref{def:locmon}) and special multiserial algebras. Among other things, we describe the ramif\mbox ications graph of symmetric special biserial algebras and we classify which of them are UMP algebras in terms of their bound quivers and their associated Brauer graphs.\\
\end{abstract}

\section{Introduction}
Since its introduction in 1987 by Assem-Skowro\'nski in \cite{AS}, gentle algebras have played an important role in the representation theory of algebras. In recent years their derived categories have been extensively studied (e.g. \cite{LP}, \cite{OPS}), in particular, the description of the indecomposable objects is well known. More precisely, Bekkert-Merklen \cite{Be-Me} gave an explicit description of the indecomposable objects in the bounded derived category of a gentle algebra. Using similar techniques as in the latter paper, in \cite{FGR1} the authors gave a combinatorial description of a family of indecomposable objects for a class of algebras, which they called SUMP algebras. The name \textit{SUMP algebra} is the abbreviation for a String algebra $A=\Bbbk Q/I$ with the Unique Maximal Path property, which means that every arrow extends to a unique maximal path in $A$. Special cases of SUMP algebras are gentle algebras and SAG algebras, where the name of the last class is an abbreviation for String Almost Gentle algebras, as introduced in \cite{FGR1}. Furthermore, the authors in \cite{FGR1} conjectured that string algebras with certain cyclic relations generating their admissible ideals determine the class of SUMP algebras (see Example 16 and Def\mbox inition 17 in \cite{FGR1}). In the present paper we provide an answer to this conjecture in a more general setting. 

It is worth mentioning that maximal paths occur in the study of derived equivalences between bound quiver algebras (see e.g. \cite{AG}, \cite{WX}). Also, for bound quiver algebras the Unique Maximal Path property, as far as we know, is also a very relevant and useful property involved in classif\mbox ication problems of certain objects in important categories, in addition to those introduced above for their derived categories. Indeed, on one hand in \cite{Gr-Sc}, this property allowed the authors to identify and classify relevant properties of the almost gentle algebras, their trivial extensions and their derived categories. In the same direction, in \cite{FM} the author determined the derived representation type of certain classes of quadratic string algebras, which are examples of SUMP algebras. On the other hand in \cite{GRV}, the Unique Maximal Path property is indispensable to determine the Auslander--Reiten component of string (indecomposable) complexes for certain classes of symmetric special biserial algebras.

The former conjecture and the latter digression inspired us to search necessary and suf\mbox f\mbox icient conditions in order to characterize bound quiver algebras satisfying the Unique Maximal Path property, brief\mbox ly, UMP algebras. That is, a UMP algebra is a bound quiver algebra such that two dif\mbox ferent maximal paths have no common arrows. Examples of UMP algebras, additional to SUMP algebras, are the Almost Gentle algebras in \cite{Gr-Sc}, as showed in \cite[Lemma 12]{FGR1}. In this paper, as our f\mbox irst attempt to achieve this classif\mbox ication, we provide necessary and suf\mbox f\mbox icient conditions for a \textit{locally monomial} (see Def\mbox inition \ref{def:locmon}) special multiserial algebra to be a UMP algebra. More precisely, our main result, which we will prove as Theorem \ref{thm:main} in Section 4, is the following.

\begin{theorem}
Let $A=\Bbbk Q/I$ be a special multiserial and locally monomial algebra. Then, $A$ is a UMP algebra if and only if for every $\omega$-relation $r$, the path $\omega(N(r))$ is cyclic and $r$ is the unique relation in $R_{N(r)}$ that is a subpath of some power of $\omega(N(r))$, i.e.
\begin{equation*}
R_{N(r)}\cap\{\text{subpaths of}\,\, \omega(N(r))^{(p)}:\,p\in\mathbb{Z}^+\}=\{r\}.
\end{equation*}
\end{theorem}

When applied to special multiserial and monomial algebras, Theorem \ref{thm:main} yields a new, more manageable presentation (see Corollary \ref{cor:main}). Actually, this characterization for monomial algebras can be extended to a broader class of algebras as presented in Corollary \ref{cor:conesp}. Moreover, Proposition \ref{prop:practice}  illustrates how these results can be used in practice to study the UMP property in concrete cases (see also Remark \ref{rmk:practice}).

An important class of f\mbox inite-dimensional algebras is the class of symmetric special biserial algebras. It is known that these algebras can be obtained as Brauer graph algebras. As an application of Theorem \ref{thm:main} and the techniques developed in this paper, we prove that the weakly connected components of the ramif\mbox ications graph for symmetric special biserial algebras, or equivalently, for Brauer graph algebras, are in bijection with the special cycles of the Brauer graph associated with the Brauer graph algebra (see Theorem \ref{thm:Bralg}). This correspondence establishes a connection between the combinatorial tools of these algebras (see Theorems \ref{thm:trunc} and \ref{ubg}) and the UMP algebras. Motivated by this correspondence, we propose that the ramif\mbox ications graph can be considered as a useful tool that concentrates the essential combinatorial features of any algebra. The constructions and techniques introduced in this paper pave the way for exploring new directions to investigate and classify the most relevant properties of an algebra and its representations.

This article is organized as follows. In Section 2, we f\mbox ix some notations that will be used along this article as well as some preliminary properties of bound quiver algebras. In Section 3, we introduce the class of UMP algebras and some preliminary results in order to achieve the necessary and suf\mbox f\mbox icient conditions for a locally monomial and special multiserial algebra to be a UMP algebra. Section 4 is dedicated to a detailed introduction and study of the principal tools, including the \emph{$\omega$-relations} and maximal paths for each weakly connected component. These concepts are essential for proving the central result of this paper, which concerns the characterization mentioned above. Finally, in Section 5, we show some important consequences of the classif\mbox ication theorem for Brauer graph algebras. In particular, we prove that the ramif\mbox ications graph, along with several derived properties, can be used to uncover signif\mbox icant characteristics and essential combinatorial features of symmetric special biserial algebras. Through examples in this section, we also explore situations in which the conditions of our main results and the assumptions of our tools are challenged.

Along the text we have included illustrative examples for most concepts and def\mbox initions to aid the reader's understanding and practical use of the material. We refer the reader to look at e.g. \cite{E}, \cite{GS}, \cite{Gr-Sc} and the references therein in order to get further information regarding concepts from special biserial (as e.g. gentle), special multiserial algebras, their representation theory and their derived categories.

\section{Preliminaries}\label{Sec-Preliminaries}

Let $A$ be a f\mbox inite-dimensional algebra of the form $\Bbbk Q/I$ over an algebraically closed f\mbox ield $\Bbbk$ of arbitrary characteristic, where $Q$ is a f\mbox inite and connected quiver, and $I$ is an admissible ideal of $\Bbbk Q$. Let $R$ be a minimal set of relations such that $I=\langle R \rangle$. As usual, we denote by $Q_0$ (resp. by $Q_1$) the set of vertices (resp. the set of arrows) of $Q$. Also, $s(\alpha)$ (resp. $t(\alpha)$) denotes the vertex of $Q_0$ where the arrow $\alpha$ starts (resp. ends).

We denote by $e_i$ the trivial path at vertex $i\in Q_0$ and by $\mathcal{P}(Q)$ the set of all paths in $Q$. Notice that $\mathcal{P}(Q)$ is a semigroup where the operation is the concatenation of paths which satisf\mbox ies the cancellation property. Throughout the paper we will use the following terminology: given $u,v\in \mathcal{P}(Q)$, we say that $u$ divides $v$ (or $u$ is a divisor of $v$, or $u$ is a factor of $v$, or $v$ is factored by $u$), which is denoted by $u\mid v$ if and only if $u$ is a subpath of $v$. We also say that two paths are \textit{disjoint} if they have no common non-trivial divisors. A path is say to be \textit{repetition-free} if it has no repeated arrows as factors.

Let $\mathfrak{m}\in\mathcal{P}(Q)$. We say that $\mathfrak{m}+I$ is a \textit{maximal path} of $A=\Bbbk Q/I$ if $\mathfrak{m}\notin I$ and for every arrow $\alpha\in Q_1$ we have $\alpha \mathfrak{m}\in I$ and $\mathfrak{m}\alpha\in I$. We denote by $\mathcal{M}$ the set of maximal paths of $A$. We also say that two maximal paths $\mathfrak{m}+I$ and $\mathfrak{m}'+I$ are disjoint if for every pair of representatives $w$ of $\mathfrak{m}+I$ and $w'$ of $\mathfrak{m}'+I$, with $w,w'\in\mathcal{P}(Q)$, we have that $w$ and $w'$ are disjoint paths.

Given $w\in\mathcal{P}(Q)$, we write $w=w^0w^1\cdots w^{l_w}$ for the factorization of $w$ in terms of the arrows $w^0,w^1\dots, w^{l_w}\in Q_1$. Thus, the length of $w$ is $l_w+1$. Also, we set $s(w):=s(w^0)$ and $t(w):=t(w^{l_w})$. We def\mbox ine the following sets
$$
E_w:=\left\{u\in\mathcal{P}(Q)\mid w=uv, \ \text{ for some }   v\in\mathcal{P}(Q)\right\}$$
and $$T_w:=\left\{u\in\mathcal{P}(Q)\mid w=vu, \ \text{ for some }   v\in\mathcal{P}(Q)\right\}.
$$

\section{UMP algebras}\label{Sec:UMP algebras}

In this section we introduce the class of UMP algebras, which are algebras satisfying a very special condition. That is, in a nutshell, a UMP algebra is a bound quiver algebra $\Bbbk Q/I$ such that two dif\mbox ferent maximal paths in $\mathcal{M}$ have no common arrows (they are disjoint). We call this condition the \textit{Unique Maximal Path} property - UMP for short. In \cite{FGR1}, A. Franco, H. Giraldo and P. Rizzo presented the class of SUMP algebras which is a class of string algebras verifying the UMP property. This property turned out to be very important to describe combinatorially some indecomposable objects in the derived categories of such algebras, that is, this property allowed these authors to apply the same techniques developed by V. Bekkert and H. Merklen in \cite{Be-Me}, but in the more general case of SUMP algebras. We observe that there are algebras, which are not string algebras, such that the UMP property holds. This motivates the following def\mbox inition.

\begin{definition}\label{Def:UMP algebras}
Let $A=\Bbbk Q/I$ be a bound quiver algebra over an algebraically closed f\mbox ield $\Bbbk$. The algebra $A$ is a \emph{UMP algebra} if it satisf\mbox ies the \emph{Unique Maximal Path property}, i.e., every arrow in $Q_1$ can be extended to a unique maximal path of $\Bbbk Q/I$.
\end{definition}

In this paper we attempt to give necessary and suf\mbox f\mbox icient conditions for a locally monomial (see Def\mbox inition \ref{def:locmon}) special multiserial algebra to be a UMP algebra. In order to achieve this, we present f\mbox irst some preliminary results following the notations introduced in the previous section.

Following \cite[\S IV.6]{SY}, we say that $Q$ is a \textit{cyclic quiver} if $Q$ is as in Figure \ref{cycqu} for $k\geq 0$.
\begin{figure}[h!]
\mypictureb[baseline=-22mm]
\caption{A cyclic quiver $Q$.}\label{cycqu}
\end{figure}
In particular, if $k=0$, then $Q_0$ consists of a unique vertex and $Q_1$ consists of a loop incident to such vertex.

\begin{definition}\label{dwau}
For each arrow $a\in Q_1$ we def\mbox ine a path $\omega_a$ in $\mathcal{P}(Q)$ as follows.
\begin{enumerate}[$i)$]
\item If $Q$ is a cyclic quiver, we f\mbox ix a repetition-free path of the form $\alpha_0\cdots \alpha_k$ in $\mathcal{P}(Q)$, where $\alpha_0,\ldots,\alpha_k$ are all the arrows of $Q$ (see e.g. Figure \ref{cycqu}). For each arrow $a\in Q_1$ we def\mbox ine $\omega_a:=\alpha_0\cdots \alpha_k$.
\item If $Q$ is not a cyclic quiver, we def\mbox ine $\omega_a$ in the following fashion: We def\mbox ine $\Phi_a$ as the set of all paths $\omega$ satisfying the following conditions.
\begin{enumerate}[$a)$]
\item $a$ is a subpath of $\omega$.
\item If $l_{\omega}>0$, then it holds that
\begin{equation}\label{ineqs}
    |\{\alpha\in Q_1\mid s(\alpha)=i\}|= 1 \ \text{  and } \ |\{\alpha\in Q_1 \mid t(\alpha)=i\}|= 1
\end{equation}
for every vertex $i\in\{t(\omega^j):0\leq j<l_{\omega}\}$.
\end{enumerate}
Then, we def\mbox ine $\omega_a$ as the path of maximal length in $\Phi_a$.
\end{enumerate}
\end{definition}

\begin{remark}\label{rnc}
\begin{enumerate}[$i)$]
\item For each arrow $a\in Q_1$, the path $\omega_a$ is well-def\mbox ined. To see this, we consider the following cases.\\
\textbf{Case 1.} If $Q$ is a cyclic quiver, the path $\omega_a$ is f\mbox ixed for all the arrows $a\in Q_1$, so we are done.\\
\textbf{Case 2.} Assume that $Q$ is not a cyclic quiver. Note that $\Phi_a\neq \emptyset$ because $a\in \Phi_a$. Moreover, the length of such paths is bounded because $Q$ is f\mbox inite and not cyclic. Suppose that $\omega$ and $\omega'$ are two di\mbox f\mbox ferent paths of maximal length in $\Phi_a$. Let $x$ be the greatest subpath of both $\omega$ and $\omega'$ such that $a$ is a subpath of $x$. Write $\omega=uxv$ and $\omega'=u'xv'$ for some paths $u,v,u',v'$. Then, at least one of the paths $u$, $v$, $u'$ or $v'$ is non-trivial. Without loss of generality, suppose that $u$ is non-trivial. Since $t(u)=t(u')=s(x)$, the condition $b)$ in Def\mbox inition \ref{dwau} and the choice of $x$ imply that $u'$ is trivial, and hence $uxv'$ is a longer path than $\omega'$ satisfying the conditions $a)$ and $b)$. This is a contradiction with the choice of $\omega'$. 
\item The path $\omega_a$ is repetition-free. Let's verify this claim as follows. If $Q$ is cyclic, the claim follows directly from Def\mbox inition \ref{dwau}$.i)$. Assume that $Q$ is not cyclic and suppose that $\omega_a$ is not a repetition-free path. Then, $\omega_a=\omega \alpha \omega'\alpha \omega''$ for some $\omega,\omega',\omega''\in \mathcal{P}(Q)$ and $\alpha\in Q_1$. Since $\omega_a$ satisf\mbox ies condition $b)$, it follows that, for every vertex $j$ in the cyclic subpath $\alpha \omega'$ of $\omega_a$, there extist exactly one arrow starting at $j$ and exactly one arrow ending at $j$. Hence, $Q$ is cyclic, which is a contradiction. This proves that $\omega_a$ is a repetition-free path.
\end{enumerate}
\end{remark}

Roughly speaking, $\omega_a$ can be interpreted as the repetition-free path of maximal length in $Q$ that contains the arrow $a$ and has no {\it ramif\mbox ications}, i.e., vertices where at least two arrows end and/or start. For simplicity, we write $l_a$ instead of $l_{\omega_a}$, where $l_{\omega_a}+1$ is the length of the path $\omega_a$. Note that if $a$ and $b$ are arrows of $Q$, then $b$ is a subpath of $\omega_a$ if and only if $\omega_a=\omega_b$.

\begin{exam}\label{exwa}

\begin{enumerate}[$i)$]

\item Consider the following quiver.
$$\xymatrix{& & \cdot\ar[dl]_{a} & & & & & & \\ 
Q: &\cdot\ar[rd]_b &  &\cdot \ar[ul]_{d}\ar@<0.7ex>[rr]^{f} \ar@<-0.5ex>[rr]_{e} & & \cdot \ar[rr]_g & & \cdot \ar[rr]_h & & \cdot \\
& & \cdot \ar[ur]_{c} & & & & & & }$$
Then, we have that $\omega_a=\omega_b=\omega_c=\omega_d=dabc$, $\omega_e=e$, $\omega_f=f$, and $\omega_g=\omega_h=gh$.

\item Consider the following quiver.
$$Q:
\ \xymatrix{\cdot \ar@(ru,lu)[]_a \ar@(ld,rd)[]_b \ar[rr]^c & & \cdot \ar@<-0.5ex>[rr]_d & & \cdot \ar@<-0.5ex>[ll]_e}$$.
Then, we have that $\omega_a=a$, $\omega_b=b$, $\omega_c=c$ and $\omega_d=\omega_e=de$.

\item Consider the following quiver.
$$Q:
\xymatrix{\cdot \ar@<-0.5ex>[rr]_b & & \cdot \ar@<-0.5ex>[ll]_a \ar@<-0.5ex>[rr]_e \ar@<-0.5ex>[d]_c & & \cdot \ar@<-0.5ex>[ll]_f\\
& & \cdot \ar@<-0.5ex>[u]_d & & }$$
Then, we have that $\omega_a=\omega_b=ab$, $\omega_c=\omega_d=cd$, and $\omega_e=\omega_f=ef$.

\end{enumerate}

\end{exam}

\begin{definition}\label{def:ramgraph}
Let $A=\Bbbk Q/I$ be a bound quiver algebra. We denote by $G_{Q,I}$ the oriented graph $(V,E)$, which we call \emph{the ramif\mbox ications graph associated to $(Q,I)$}, putting the set of vertices $V=\{\omega_a\mid\,a\in Q_1\}$ and for any pair $a,b\in Q_1$, there exists a directed edge $\delta\in E$ from $\omega_a$ to $\omega_b$ if and only if $\omega_a\neq\omega_b$, $t(\omega_a)=s(\omega_b)$ and $\omega_a^{l_a}\omega_b^{0}\notin I$.
\end{definition}

Observe that, from Def\mbox inition \ref{def:ramgraph}, directly follows that $G_{Q,I}$ has no loops.

In this article, we will understand by a {\it weakly connected component} of (the oriented graph) $G_{Q,I}$ the subgraph of $G_{Q,I}$ whose underlying graph, obtained by ignoring the orientations of the edges in $G_{Q,I}$, is a connected component of this latter graph. We denote by $\mathcal{D}_{Q,I}$ the set of weakly connected components of $G_{Q,I}$. For each $N\in \mathcal{D}_{Q,I}$, we denote by $Q_N$ the subquiver of $Q=(Q_0,Q_1,s,t)$ def\mbox ined by the paths $\omega_a$ which are vertices in the component $N$. More precisely, $Q_N=((Q_N)_0,(Q_N)_1,s_N,t_N)$ where the set of arrows is def\mbox ined by
$$
(Q_N)_1:=\{\alpha\in Q_1\mid\,\omega_{\alpha}\,\,\text{is a vertex in}\,\,N\},
$$
the set of vertices is def\mbox ined by
$$
(Q_N)_0:=\{i\in Q_0\mid\,i\in\{s(\alpha),t(\alpha)\}\,\text{for some}\, \alpha \in (Q_N)_1\},
$$
and $s_N:=s|_{(Q_N)_1}$ and $t_N:=t|_{(Q_N)_1}$. In consequence, from the subalgebra $\Bbbk Q_N$ of $\Bbbk Q$, we denote by $I_N:=I\cap\Bbbk Q_N$ the induced ideal, which allows us to def\mbox ine the algebra $A_N:=\Bbbk Q_N/I_N$ and the set of maximal paths $\mathcal{M}_N$ in $A_N$. An important result, which will establish some conditions for our classif\mbox ication, is as follows:

\begin{proposition}\label{prop:monoalg}
Let $I$ be an admissible ideal of the path algebra $\Bbbk Q$. Then the ideal $I_N$ is an admissible ideal of $\Bbbk Q_N$ for all $N\in\mathcal{D}_{Q,I}$.
\end{proposition}
\begin{proof}
Since $I$ is an admissible ideal, then $\mathcal{R}_{Q}^m\subseteq I\subseteq \mathcal{R}_{Q}^2$, for some $m\geq2$, where $\mathcal{R}_Q$ is the so-called \textit{arrow ideal} of $\Bbbk Q$. Hence, as $\mathcal{R}_{Q_N}^m=\mathcal{R}_{Q}^m\cap \Bbbk Q_N$, we obtain $\mathcal{R}_{Q_N}^m\subseteq I_N\subseteq\mathcal{R}_{Q_N}^2$, which proofs the proposition. 
\end{proof}

\begin{exam}\label{exgdr}

\begin{enumerate}[$i)$]

\item In the Example \ref{exwa}.$i)$, take $I=\langle cd,cf,abc,eg\rangle$. We have the following ramif\mbox ications graph.
$$G_{Q,I}: \ \xymatrix{\cdot_{\omega_a} \ar[rr] &  & \cdot_{\omega_e} \\ \ \ \cdot_{\omega_f} \ar[rr] & & \cdot_{\omega_g}}$$
\vspace{5mm}

\noindent
For this graph, $\mathcal{D}_{Q,I}=\{L,N\}$, where $\xymatrix{L:\ \cdot_{\omega_a} \ar[rr] &  & \cdot_{\omega_e}}$ and $\xymatrix{N:\ \cdot_{\omega_f} \ar[rr] &  & \cdot_{\omega_g}}$. Then,
$$\xymatrix{& & \cdot\ar[dl]_{a} & & & \\ 
Q_L: & \cdot\ar[rd]_b &  & \cdot \ar[ul]_{d} \ar[rr]_e & & \cdot\ ,\\
& & \cdot \ar[ur]_{c} & & & }\ \xymatrix{ \\ I_L=\langle cd, abc\rangle,\ \mathcal{M}_L=\{dab, bce\}}\\ $$
$$Q_N: \xymatrix{\cdot \ar[rr]^f & & \cdot \ar[rr]^g & & \cdot \ar[rr]^h & & \cdot \ ,} \ I_N=\langle 0\rangle,\ \mathcal{M}_N=\{fgh\}.$$

\item In the Example \ref{exwa}.$ii)$, take $I=\langle ab,ac,ba,bc,ed,a^2-b^2\rangle$. We have the following ramif\mbox ications graph.
$$G_{Q,I}: \ \xymatrix{\cdot_{\omega_a} &  & \cdot_{\omega_b} \\ \ \ \cdot_{\omega_c} \ar[rr] & & \cdot_{\omega_d}}$$
Hence, $\mathcal{D}_{Q,I}=\{N_1,N_2,N_3\}$, where $N_1:\ \cdot_{\omega_a}$, $N_2:\ \cdot_{\omega_b}$, and $N_3:\ \xymatrix{ \cdot_{\omega_c} \ar[rr] & & \cdot_{\omega_d} }$. Then,
$$Q_{N_1}: \xymatrix{\cdot \ar@(ur,ul)[]_a},\ \ I_{N_1}=\langle a^3\rangle,\ \mathcal{M}_{N_1}=\{a^2\},$$
$$Q_{N_2}: \xymatrix{\cdot \ar@(ur,ul)[]_b},\ \ I_{N_2}=\langle b^3\rangle,\ \mathcal{M}_{N_2}=\{b^2\},$$
$$Q_{N_3}: \xymatrix{\cdot \ar[rr]^c & & \cdot \ar@<-0.5ex>[rr]_d & & \cdot \ar@<-0.5ex>[ll]_e},\ \ I_{N_3}=\langle ed\rangle,\ \mathcal{M}_{N_3}=\{cde\}.$$

\item In the Example \ref{exwa}.$iii)$, take $I=\langle ba,be,dc,de,fa,fc,abcd-ef\rangle$. We have the following ramif\mbox ications graph.
$$G_{Q,I}:\ \xymatrix{
\cdot_{\omega_a} \ar@<-0.5ex>[rr] & & \cdot_{\omega_c} \ar@<-0.5ex>[ll]  & \cdot_{\omega_e}}$$
Hence, $\mathcal{D}_{Q,I}=\{L,N\}$, where $L:\ \xymatrix{ \cdot_{\omega_a} \ar@<-0.5ex>[rr] & & \cdot_{\omega_c} \ar@<-0.5ex>[ll]}$, and $N:\ \cdot_{\omega_e}$. Then,
$$Q_L: \xymatrix{\cdot \ar@<-0.5ex>[rr]_b & & \cdot \ar@<-0.5ex>[ll]_a \ar@<-0.5ex>[rr]_c & & \cdot \ar@<-0.5ex>[ll]_d},\ I_L=\langle ba,dc,abcda,dabcd\rangle,\ \mathcal{M}_L=\{abcd,bcdabc\},$$
$$Q_N: \xymatrix{\cdot \ar@<-0.5ex>[rr]_e & & \cdot \ar@<-0.5ex>[ll]_f},\ I_N=\langle efe,fef\rangle,\ \mathcal{M}_N=\{ef,fe\}.$$

\end{enumerate}

\end{exam}

\begin{remark}\label{rem:paths1}
Any non-trivial path $u$ in $\mathcal{P}(Q)$ is a subpath of $\omega_{\alpha_1}\cdots\omega_{\alpha_k}$, for some $\alpha_1,\ldots,\alpha_k\in Q_1$. We show this as follows: If $Q$ is cyclic, this claim is true because every path is a subpath of a power of the unique path of the form $\omega_a$ in $Q$, with $a\in Q_1$. Assume that $Q$ is not cyclic. Recall that we write $u=u^0\cdots u^{l_u}$ for the decomposition of $u$ as product of arrows. We def\mbox ine
$$
H=\{t(u^j):0\leq j<l_{u},\mbox{ and }|\{\alpha\in Q_1\mid s(\alpha)=t(u^j)\}|\neq 1\mbox{ or }|\{\alpha\in Q_1 \mid t(\alpha)=t(u^j)\}|\neq 1\}.
$$
If $H=\emptyset$, then $u\in \Phi_{u^0}$. Hence, $u$ is a subpath of $\omega_{u^0}$, so we are done. Now if $H\neq\emptyset$, then there exist $m>0$ and $0\leq j_1<\cdots<j_m< l_u$ such that $H=\{t(u^{j_1}),t(u^{j_2}),\ldots, t(u^{j_m})\}$ and, consequently, $u$ is a subpath of $\omega_{u^{j_1}}\omega_{u^{j_2}}\cdots\omega_{u^{l_u}}$.
\end{remark}

\begin{lemma}\label{lem:components}
Let $u$ be a non-trivial path in $\mathcal{P}(Q)$ such that there are no relations of the form $\omega_a^{l_a}\omega_b^0\in I$ for which $\omega_a^{l_a}\omega_b^0\mid u$. Then, $u\in\mathcal{P}(Q_N)$ for a unique component $N\in \mathcal{D}_{Q,I}$.
\end{lemma}
\begin{proof}
From Remark \ref{rem:paths1}, we have that $u$ is a subpath of a path of the form $\omega_{\alpha_1}\cdots\omega_{\alpha_k}$, for some $\alpha_1,\ldots,\alpha_k\in Q_1$. If $k=1$, then $u$ is a path in the quiver induced by the weakly connected component that contain the vertex $\omega_{u^0}$. If $k>1$, there is an edge from $\omega_{\alpha_i}$ to $\omega_{\alpha_{i+1}}$ in the graph $G_{Q,I}$, because  $\omega_{\alpha_i}^{l_i}\omega_{\alpha_{i+1}}^0\notin I$ for all $i\in\{1,\ldots,k-1\}$. Hence, there exists a unique weakly connected component $N$ containing $\omega_{\alpha_1},\ldots,\omega_{\alpha_k}$ as vertices and thus $u\in\mathcal{P}(Q_N)$.
\end{proof}

\begin{remark}\label{rem:unicomp}\

\begin{enumerate}

\item[i)] Clearly, if $u$ is a non-zero path or is a zero relation in $I$ without subpaths of the form $\omega_a^{l_a}\omega_b^0$, with $a,b\in Q_1$, then $u$ belongs to $Q_N$ for a unique $N\in\mathcal{D}_{Q,I}$. If $u$ is a path as in the hypothesis of Lemma \ref{lem:components}, then we denote by $N(u)$ the unique weakly connected component such that $u\in \mathcal{P}(Q_{N(u)})$ guaranteed by this lemma.

\item[ii)] For any $N\in\mathcal{D}_{Q,I}$, we denote by $\mathcal{M}_N$ the set of maximal paths of $A_N$. We def\mbox ine the function $f_N:\mathcal{M}_N\rightarrow A$ by $\mathfrak{m}+I_N\mapsto \mathfrak{m}+I$, which is well-def\mbox ined since $I_N\subseteq I$.
\end{enumerate}
\end{remark}

\begin{theorem}\label{thm:maxUMP}
Let $\mathcal{M}$ be the set of maximal paths of a bound quiver algebra $A=\Bbbk Q/I$. Then, under the notations in Remark \ref{rem:unicomp}\,ii), we have that $\mathcal{M}=\bigcup\limits_{N\in\mathcal{D}_{Q,I}}f_N(\mathcal{M}_N)$.
\end{theorem}
\begin{proof}
For any $\mathfrak{m}+I\in\mathcal{M}$ we will prove that $\mathfrak{m}+I_{N(\mathfrak{m})}\in\mathcal{M}_{N(\mathfrak{m})}$ (see Remark \ref{rem:unicomp}\,i)). Let $\alpha\in (Q_{N(\mathfrak{m})})_1$ such that $t_{N(\mathfrak{m})}(\alpha)=s_{N(\mathfrak{m})}(\mathfrak{m})$. Since $\mathfrak{m}$ is a maximal path, we have that $\alpha \mathfrak{m}\in I$ and, due to $\alpha\mathfrak{m}\in\mathcal{P}(Q_{N(\mathfrak{m})})$, we conclude $\alpha \mathfrak{m}\in I_{N(\mathfrak{m})}$. Similarly, $\mathfrak{m}\alpha\in I_{N(\mathfrak{m})}$, if $s_{N(\mathfrak{m})}(\alpha)=t_{N(\mathfrak{m})}(\mathfrak{m})$. Then, $\mathfrak{m}+I_{N(\mathfrak{m})}\in\mathcal{M}_{N(\mathfrak{m})}$ and, consequently, $\mathfrak{m}+I=f_{N(\mathfrak{m})}(\mathfrak{m}+I_{N(\mathfrak{m})})\in f_{N(\mathfrak{m})}(\mathcal{M}_{N(\mathfrak{m})})$. 

Conversely, we consider $N\in\mathcal{D}_{Q,I}$, $\mathfrak{m}+I_{N}\in M_{N}$ and $\alpha\in Q_1$. We will prove that $\alpha \mathfrak{m}\in I$. Suppose, by the contrary, that $\alpha \mathfrak{m}\notin I$. Then, $\alpha m^0\notin I$. In consequence  $\omega_{\alpha}=\omega_{m^0}$ or there is an edge from $\omega_{\alpha}$ to $\omega_{m^0}$ in the graph $G_{Q,I}$. In any case, $\omega_{\alpha}$ and $\omega_{m^0}$ are in the same weakly connected component of $G_{Q,I}$ and thus, $N(\alpha)=N(m^0)=N(m)=N$.
In particular, $\alpha\in(Q_N)_1$. Now, since $\mathfrak{m}+I_N\in\mathcal{M}_N$, we have that $\alpha\mathfrak{m}\in I_N\subseteq I$, which contradicts the fact that $\alpha\mathfrak{m}\notin I$. Hence, $\alpha\mathfrak{m}\in I$. Analogously, we have that $\mathfrak{m}\alpha\in I$. Thus, $f_N(\mathfrak{m}+I_N)=\mathfrak{m}+I\in\mathcal{M}$, i.e., $\bigcup\limits_{N\in\mathcal{D}_{Q,I}}f_N(\mathcal{M}_N)\subseteq \mathcal{M}$, which proves the theorem.
\end{proof}

\begin{corollary}\label{cor:local}
Let $A=\Bbbk Q/I$ be a bound quiver algebra. Then, $A$ is a UMP algebra if and only if $A_N$ is a UMP algebra for all $N\in \mathcal{D}_{Q,I}$.
\end{corollary}
\begin{proof}
Clearly, if $A$ is a UMP algebra, by Theorem \ref{thm:maxUMP}, $A_N$ is a UMP algebra for every component $N\in\mathcal{D}_{Q,I}$. Conversely, suppose that $A$ is not a UMP algebra. Then, there exist $\mathfrak{m}+I,\mathfrak{n}+I\in\mathcal{M}$ such that $\alpha\mid\mathfrak{m}$ and $\alpha\mid\mathfrak{n}$, for some $\alpha\in Q_1$. Since $\alpha$ belongs to $(Q_N)_1$ for a unique $N\in\mathcal{D}_{Q,I}$, we obtain that $N(\mathfrak{m})=N(\alpha)=N(\mathfrak{n})=N$. In consequence, $\mathfrak{m}+I$ and $\mathfrak{n}+I$ belongs to $\mathcal{M}_N$, that is, $A_N$ is not a UMP algebra.
\end{proof}

\begin{exam}\label{eanum} In the following examples, we show how the algebras in Example \ref{exgdr} ilustrates Theorem \ref{thm:maxUMP} and Corollary \ref{cor:local}.
\begin{enumerate}[$i)$]
\item For the Example \ref{exgdr}.$i)$, we have that $\mathcal{M}=\{dab, bce,fgh\}$, $\mathcal{M}_L=\{dab, bce\}$ and $\mathcal{M}_N=\{fgh\}$. Here, we observe that $A$ is not UMP in accordance to the fact that $A_L$ is not UMP.
\item For the Example \ref{exgdr}.$ii)$, we have that $\mathcal{M}=\{a^2,cde\}$, $\mathcal{M}_{N_1}=\{a^2\}$, $\mathcal{M}_{N_2}=\{b^2\}$ and $\mathcal{M}_{N_3}=\{cde\}$, where we note that $A$ is UMP, according to the fact that all the algebras $A_{N_i}$ are UMP for $i\in \{1,2,3\}$.
\item For the Example \ref{exgdr}.$iii)$, we have that $\mathcal{M}=\{abcd,bcdabc,ef,fe\}$, $\mathcal{M}_L=\{abcd,bcdabc\}$ and $\mathcal{M}_N=\{ef,fe\}$, where we note that $A$ is not UMP, according to the fact that neither $A_L$ nor $A_N$ are UMP.
\end{enumerate}
\end{exam}

\section{The main result}
In this section we will assume that the bound quiver algebra $A=\Bbbk Q/I$ is a \emph{special multiserial algebra}. Precisely,  following \cite[Def\mbox inition 2.2]{GS}, a special multiserial algebra is a bound quiver algebra in which the following property holds: for every arrow $\alpha$ in $Q_1$ there is at most one arrow $\beta$ in $Q_1$ such that $\alpha\beta\notin I$ and at most one arrow $\gamma$ in $Q_1$ such that $\gamma\alpha\notin I$.

\begin{lemma}\label{lem:compchar}
Let $A=\Bbbk Q/I$ be a special multiserial algebra and let $N\in\mathcal{D}_{Q,I}$ be a weakly connected component of the ramif\mbox ications graph $G_{Q,I}$. Then $N$ has only one of the following forms, for some $n\in\mathbb{N}_{\geq0}$:
\begin{figure}[!htb]
   \begin{minipage}{0.48\textwidth}
\centering
 $ N:=\,\xymatrix{
\omega_0 \ar[r] & \omega_1\ar[r]  & \cdots\ar[r] & \omega_n}$
   \end{minipage}\hfill
   \begin{minipage}{0.48\textwidth}
 \centering
 $N:=\,\xymatrix{
\omega_0 \ar[r] & \omega_1\ar[r]  & \cdots\ar[r]
& \omega_n \ar@/^2pc/[lll]}$
\end{minipage}
\end{figure}
\end{lemma}
\begin{proof}
Since $A$ is a special multiserial algebra, for any arrow  $\alpha$ in $Q_1$ there is at most one arrow $\beta\in Q_1$ such that $\omega_{\alpha}^{l_{\alpha}}\omega_{\beta}^0\notin I$ and  there is at most one arrow $\gamma\in Q_1$ such that $\omega_{\gamma}^{l_{\gamma}}\omega_{\alpha}^0\notin I$. Hence, we have that at most one edge in $G_{Q,I}$ ends in $\omega_{\alpha}$ and at most one edge in $G_{Q,I}$ starts in $\omega_{\alpha}$, which completes the proof.
\end{proof}
From the notation of Lemma \ref{lem:compchar}, we def\mbox ine the path $\omega(N):=\omega_0\cdots\omega_n$. The trivial path at the vertex $t(\omega(N))$ will be denoted by $\omega(N)^{(0)}$ and, for any $l\in\mathbb{Z}^+$, we denote by $\omega(N)^{(l)}$ the composition of $\omega(N)$ with itself $l$-times. For convenience, the subindex $i$ corresponding to the path $\omega_i$ will be assumed as a representative of a class in $\mathbb{Z}/(n+1)\mathbb{Z}$. On the one hand, if $n>0$ and $\omega_n^{l_n}\omega_0^0\notin I_N$, then $\omega(N)$ is unique up to rotations. On the other hand, $\omega(N)$ is repetition-free since the $\omega_i$'s are repetition-free and pairwise disjoint. Clearly, the unique factors in $\omega(N)$ of the form $\omega_a^{l_a}\omega_b^0$ are the paths $\omega_i^{l_i}\omega_{i+1}^0$. Likewise, observe that $\omega_i^{l_i}\omega_{i+1}^0\notin I_N$ for all $i<n$.

The path $\omega(N)$ will allow us to determine the set of maximal paths $\mathcal{M}_N$ in $A_N$. Consequently, it will allow us to provide a characterization for $A_N$ to be a UMP algebra.

\begin{theorem}\label{thm:omegaN}
Let $A=\Bbbk Q/I$ be a special multiserial algebra, $N\in\mathcal{D}_{Q,I}$ and $u$ a non-trivial path in $\mathcal{P}(Q_N)$. Then, $u\mid \omega(N)^{(l)}$ for some $l>0$ if and only if for each factor of $u$ of the form $\omega_{\alpha}^{l_{\alpha}}\omega_{\beta}^0$, there exists $i$ such that $\omega_{\alpha}^{l_{\alpha}}\omega_{\beta}^0=\omega_i^{l_i}\omega_{i+1}^{0}$.
\end{theorem}
\begin{proof} Suppose that $u\mid \omega(N)^{(l)}$, for some $l>0$. If $\alpha$ and $\beta$ are two arrows such that $\omega_{\alpha}^{l_{\alpha}}\omega_{\beta}^0\mid u$, then $\omega_{\alpha}^{l_{\alpha}}\omega_{\beta}^0\mid \omega(N)^{(l)}$, and hence we obtain that $\omega_{\alpha}^{l_{\alpha}}\omega_{\beta}^0=\omega_i^{l_i}\omega_{i+1}^0$, for some $i$.
For the converse, suppose that for every subpath of $u$ of the form $\omega_{\alpha}^{l_{\alpha}}\omega_{\beta}^0$, there exists $i$ such that $\omega_{\alpha}^{l_{\alpha}}\omega_{\beta}^0=\omega_i^{l_i}\omega_{i+1}^0$. Let $\omega_{\alpha_1},\ldots, \omega_{\alpha_k}$ be the paths def\mbox ined in Remark \ref{rem:paths1} $ii)$ such that $u\mid \omega_{\alpha_1}\cdots \omega_{\alpha_k}$. If $k=1$, then $u\mid \omega_{\alpha_1}$, and since $\omega_{\alpha_1}\mid \omega(N)$, it follows that $u$ is a subpath of $\omega(N)$. If $k>1$, then $\omega_{\alpha_j}^{l_{\alpha_j}}\omega_{\alpha_{j+1}}^0\mid u$ for each $j<k$ and, by our assumption, there exists $i_j$ such that $\omega_{\alpha_j}^{l_{\alpha_j}}\omega_{\alpha_{j+1}}^0=\omega_{i_j}^{l_{i_j}}\omega_{i_j+1}^0$. Hence, $\omega_{\alpha_j}=\omega_{i_j}$ for every $j\in \{1,\ldots,k\}$, and $i_{j+1}=i_j+1$ for $j<k$. Then, $\omega_{\alpha_1}\cdots \omega_{\alpha_k}=\omega_{i_1}\omega_{i_1+1}\cdots \omega_{i_1+(k-1)}$ is a subpath of $\omega(N)^{(l)}$ for some $l>0$, and so, $u\mid \omega(N)^{(l)}$.
\end{proof}

From now on, we will assume another special property of the bound quiver algebra $A=\Bbbk Q/I$, which we state in the following def\mbox inition.
\begin{definition}\label{def:locmon}
Let $A=\Bbbk Q/I$ be a bound quiver algebra. $A$ is called \emph{locally monomial} if $A_N=\Bbbk Q_N/I_N$ is a monomial algebra for all $N\in \mathcal{D}_{Q, I}$.
\end{definition}
In general, if $A=\Bbbk Q/I$ is a monomial algebra and $R$ is a minimal set of zero relations generating $I$, then $A_N=\Bbbk Q_N/I_N$ is a monomial algebra for all $N\in\mathcal{D}_{Q,I}$. Moreover, setting $R_N:=R\cap\mathcal{P}(Q_N)$, we have that $R_N$ is a minimal set of relations which generates $I_N$. Indeed, clearly $\langle R_N\rangle\subseteq\langle R\rangle\cap\mathcal{P}(Q_N)=I_N$. Conversely, if $x=\sum\limits_{p\in Q_N}\lambda_p p\in I_N$, it follows that $p\in I$ for any $p\in\mathcal{P}(Q_N)$ such that $\lambda_p\neq0$, since $x\in I$ and $A$ is a monomial algebra. Consequently, $p\in\langle R_N\rangle$ which implies that $x\in \langle R_N\rangle$. In particular, the class of monomial algebras is contained in the class of locally monomial algebras.

Theorem \ref{thm:omegaN} allows us to introduce the important sets of paths, for any $N\in\mathcal{D}_{Q,I}$:
\begin{equation}\label{eq:defT}
\mathcal{T}_N:=\{u\in\mathcal{P}(Q_N):u\,\,\text{is non-trivial and }\,\,u\mid\omega(N)^{(l)}\,\,\text{for some}\,\,l\in\mathbb{Z}^+\}    
\end{equation}
\begin{equation}\label{eq:defT2}
\Omega_N:=\{\omega_a^{l_a}\omega_b^0\in I_N:\,a,b\in (Q_N)_1\,\,\,\text{and}\,\,\,t(\omega_a)=s(\omega_b)\}\quad\text{and}\quad \mathcal{S}_N:=R_N\setminus\Omega_N.
\end{equation}
In particular, by our hypothesis over $A$ to be locally monomial, $\mathcal{S}_N$ is formed by the paths in $R_N$ which are not in $\Omega_N$. Now, by Theorem \ref{thm:omegaN}, the set $\mathcal{T}_N$ in (\ref{eq:defT}) coincides with the following set 
\begin{equation}\label{eq:def2T}
\mathcal{T}_N=\{u\in\mathcal{P}(Q_N):\,(\forall\omega_{\alpha}^{l_{\alpha}}\omega_{\beta}^0\in I_N)(\omega_{\alpha}^{l_{\alpha}}\omega_{\beta}^0\mid u\Longrightarrow\omega_{\alpha}=\omega_n\,\,\text{and}\,\,\omega_{\beta}=\omega_0)\}    
\end{equation}

Hence, under this characterization of $\mathcal{T}_N$ it is easy to see that $\mathcal{S}_N$ in (\ref{eq:defT2}) and the set $\mathcal{M}_N$ of maximal paths in $A_N$ are subsets of $\mathcal{T}_N$. We def\mbox ine $\eta:=\eta_N\in\mathbb{Z}^+$ as the least integer such that all paths in $\mathcal{S}_N$ and $\mathcal{M}_N$ are subpaths of $\omega(N)^{(\eta)}$ as in (\ref{eq:defT}).

\begin{exam}\label{exlmsmtos}

\begin{enumerate}[$i)$]

\item In the Example \ref{exgdr}.$i)$, we have $\omega(L)=\omega_a\omega_e=dabce$ and $\omega(N)=\omega_f\omega_g=fgh$. Note that $\Omega_L=\{cd\}$, $\mathcal{S}_L=\{abc\}$, $\Omega_N=\emptyset$, and $\mathcal{S}_N=\emptyset$.

\item In the Example \ref{exgdr}.$ii)$, we have $\omega(N_1)=\omega_a=a$, $\omega(N_2)=\omega_b=b$ and $\omega(N_3)=\omega_c\omega_d=cde$. Note that $\Omega_{N_1}=\emptyset$, $\mathcal{S}_{N_1}=\{a^3\}$, $\Omega_{N_2}=\emptyset$, $\mathcal{S}_{N_2}=\{b^3\}$, $\Omega_{N_3}=\{ed\}$, and $\mathcal{S}_{N_1}=\emptyset$.

\item In the Example \ref{exgdr}.$iii)$, we can take $\omega(L)=\omega_a\omega_c=abcd$ and have that $\omega(N)=\omega_e=ef$. Note that $\Omega_L=\{ba,dc\}$, $\mathcal{S}_L=\{abcda,dabcd\}$, $\Omega_N=\emptyset$, and $\mathcal{S}_N=\{efe,fef\}$.

\end{enumerate}

\end{exam}

\begin{lemma}
If $\mathcal{S}_N=\emptyset$, then $\mathcal{M}_N=\{\omega(N)\}$. In particular, $A_N$ is a UMP algebra.
\end{lemma}
\begin{proof}
If $\mathcal{S}_N=\emptyset$, then $R_N=\Omega_N$. Observe that  $\omega(N)\notin I_N$, since otherwise there is a path in $ R_N=\Omega_N$ dividing $\omega(N)$ and hence such path must be of the form $\omega_i^{l_i}\omega_{i+1}^0$, for some $i<n$, which is a contradiction because $\omega_i^{l_i}\omega_{i+1}^0\notin I_N$. For the maximality of $\omega(N)$, by Lemma \ref{lem:compchar}, it suf\mbox f\mbox ices to prove that $\omega_n^{l_n}\omega_0^0\in I_N$. We have two cases. If $I_N=\left<0\right>$, then $t_{\omega(N)}\neq s_{\omega(N)}$ due to the admissibility of $I_N$, which implies that $\omega_n^{l_n}\omega_0^0\in I_N$. On the contrary, if $I_N\neq\left<0\right>$, then by the admissibility of $I_N$, there exists $l>>0$ such that $\omega(N)^{(l)}\in I_N$. Thus, $\omega_{\alpha}^{l_{\alpha}}\omega_{\beta}^0\mid\omega(N)^{(l)}$, for some $\omega_{\alpha}^{l_{\alpha}}\omega_{\beta}^0\in I_N$. Thus, by Theorem \ref{thm:omegaN} we obtain that $\omega_{\alpha}=\omega_n$ and $\omega_{\beta}=\omega_0$ and then $\omega_{n}^{l_n}\omega_0^0\in I_N$. Finally, if $\mathfrak{m}+I_N\in \mathcal{M}_N$, then $\mathfrak{m}\mid\omega(N)$ by Theorem \ref{thm:omegaN} and the lemma follows. 
\end{proof}

\begin{definition}\label{rem:total}
Let $N\in\mathcal{D}_{Q,I}$ be a weakly connected component. We def\mbox ine on the sets $(Q_N)_1$ and $\mathcal{S}_N$ (provided that $\mathcal{S}_N\neq\emptyset$) the following orders:
\begin{itemize}
\item {\bf On $(Q_N)_1$:} for any pair of arrows $\alpha,\beta\in (Q_N)_1$ we def\mbox ine $\alpha\leq_f \beta$ if and only if $\omega(N)=u\alpha u'$, for some $u,u'\in\mathcal{P}(Q_N)$, with $\beta\mid \alpha u'$. 
\item {\bf On $\mathcal{S}_N$:} For any pair of zero relations $u,v\in\mathcal{S}_N$ we def\mbox ine $u\leq_s v$ if and only if $u^0\leq_f v^0$.
\end{itemize}
\end{definition}

\begin{lemma}\label{lem:orders}
Under the above notations, the sets $((Q_N)_1,\leq_f )$ and $(\mathcal{S}_N,\leq_s)$ are totally ordered.
\end{lemma}
\begin{proof}
For the f\mbox irst case $((Q_N)_1,\leq_f)$, the ref\mbox lexivity and transitivity are immediate. The antisymmetry is a consequence of the unique factorization on $\mathcal{P}(Q_N)$ and the fact that $\omega(N)$ is a repetition-free path.

Now for the second set $(\mathcal{S}_N,\leq_s)$, the ref\mbox lexivity and transitivity are immediate as well. For the antisymmetry, we take $u,v\in \mathcal{S}_N$ such that $u\leq_s v$ and $v\leq_s u$. Then, $u^0\leq_f v^0$ and $v^0\leq_f u^0$, which implies that $u^0=v^0$. Let $x$ be the greatest path such that $u=xy$ and $v=xz$, for some $y,z\in\mathcal{P}(Q_N)$. We will show that either $y$ or $z$ is trivial. Assume, by contradiction, that $y$ and $z$ are non-trivial. Since $A$ is a special multiserial algebra and $y^0\neq z^0$, we have either $x^{l_x}y^0\in I_N$ or $x^{l_x}z^0\in I_N$. Without loss of generality we consider  $x^{l_x}y^0\in I_N$. Thus, by minimality of $R_N$, we obtain that $l_x=l_{y}=0$. Besides, since $y^0\neq z^0$ it follows that $u=\omega_x^{l_x}\omega_{y}^0$, which contradicts that $u\in \mathcal{S}_N$. In consequence, either $y$ or $z$ is trivial and, by minimality of $R_N$, we conclude that $u=v$.

Finally, the orders ``$\leq_f $'' and ``$\leq_s$'' are total orders because for every arrow $\alpha\in (Q_N)_1$ we have that $\alpha\mid\omega(N)$.
\end{proof}
Following Lemma \ref{lem:orders}, there exist $k\in\mathbb{Z}^+$ and relations $s_1,\ldots,s_k$ such that $\mathcal{S}_N=\{s_1,\ldots,s_k\}$, with $s_1<_s\cdots<_s s_k$. From now on, we will denote by $\sigma_i:=l_{s_i}$, for $i=1,\ldots, k$. We distinguish a new set of paths of $Q_N$ as follows.
\begin{definition}\label{def:maxS}
For $0\leq i\leq k$, we def\mbox ine the paths $\mathfrak{m}_i\in\mathcal{P}(Q_N)$ as follows:
\begin{enumerate}
\item[i)] If $0<i<k$, then $\mathfrak{m}_i=vs_{i+1}^0\cdots s_{i+1}^{\sigma_{i+1}-1}$, where $v\in\mathcal{P}(Q_N)$ is such that $\omega(N)=us_{i}^0vs_{i+1}^0v'$, for some $u,v'\in\mathcal{P}(Q_N)$. Notice that $s_i\in E_{s_i^0vs_{i+1}^0v'\omega(N)^{(\eta-1)}}$ and $s_{i+1}\in E_{s_{i+1}^0v'\omega(N)^{(\eta-1)}}$.
\item[ii)] If $\omega_n^{l_n}\omega_0^0\in I_N$, then we def\mbox ine $\mathfrak{m}_0=us_1^0\cdots s_1^{\sigma_1-1}$ and $\mathfrak{m}_k=s_k^1\cdots s_k^{\sigma_k}v$, where $u,v\in\mathcal{P}(Q_N)$ are such that $us_1\in E_{\omega(N)}$ and $s_kv\in T_{\omega(N)}$.
\item[iii)] If $\omega_n^{l_n}\omega_0^0\notin I_N$, then we def\mbox ine $\mathfrak{m}_0=\mathfrak{m}_k=vus_1^0\cdots s_1^{\sigma_1-1}$, where $u$ and $v$ are subpaths of $\omega(N)$ such that $us_1^0\in E_{\omega(N)}$ and $s_k^0v\in T_{\omega(N)}$.
\end{enumerate}
We set $\mathcal{S}^{\ast}_N:=\{\mathfrak{m}_i+I_N:\,0\leq i\leq k\}$.
\end{definition}

\begin{lemma}
Let $A$ be a special multiserial and locally monomial algebra and let $N\in\mathcal{D}_{Q,I}$. For all $i\in\{0,\ldots,k\}$ we have that $\mathfrak{m}_i\notin I_N$.
\end{lemma}
\begin{proof}
Suppose that $0<i<k$ and, by contradiction, that $\mathfrak{m}_i\in I_N$. Then, there exists $r\in R_N$ such that $r\mid \mathfrak{m}_i$. Since, by Def\mbox inition \ref{def:maxS} $i)$, $\mathfrak{m}_i\mid vs_{i+1}$ and $vs_{i+1}\mid \omega(N)^{(\eta)}$, it follows from Theorem \ref{thm:omegaN} that $r\in \mathcal{S}_N$ or $r=\omega_n^{l_n}\omega_0^{0}$. In the former case, $r=s_h$, for some $h$. Then, $s_h\mid vs_{i+1}$, {\it a fortiori}, $s_h^0\mid vs_{i+1}$. By Def\mbox inition \ref{def:maxS} $i)$, we have $s_h^0\nmid v$ and, consequently, $s_h\mid s_{i+1}$. By the minimality of $R_N$, we obtain $h=i+1$. Hence, $s_{i+1}\mid \mathfrak{m}_i=vs_{i+1}^0\cdots s_{i+1}^{\sigma_{i+1}-1}$ and, by a similar argument as above, we conclude $s_{i+1}\mid s_{i+1}^0\cdots s_{i+1}^{\sigma_{i+1}-1}$, which is impossible. In the latter case, since $\omega_n^{l_n}\omega_0^0\nmid s_{i+1}$, we obtain that $s_{i+1}\mid\omega(N)$ and, consequently, $\mathfrak{m}_i\mid vs_{i+1}^0\cdots s_{i+1}^{\sigma_{i+1}}$ and $vs_{i+1}^0\cdots s_{i+1}^{\sigma_{i+1}}\mid\omega(N)$, which is a contradiction with the fact that $\omega_n^{l_n}\omega_0^0$ is not a factor of $\omega(N)$. In conclusion, $\mathfrak{m}_i\notin I_N$ for all $i\in\{1,\ldots, k-1\}$.

Finally, to prove that $\mathfrak{m}_0\notin I_N$ and $\mathfrak{m}_k\notin I_N$, following Def\mbox inition \ref{def:maxS} $ii), iii)$, we need to analyze two cases: $\omega_n^{l_n}\omega_0^0\in I_N$ or $\omega_n^{l_n}\omega_0^0\notin I_N$. In the former case, we obtain $\mathfrak{m}_0\mid \omega(N)$ and $\mathfrak{m}_k\mid \omega(N)$. Now, if $\mathfrak{m}_0\in I_N$, then, as above, there exists $s_h\in \mathcal{S}_N$ such that $s_h\mid \mathfrak{m}_0$ and, hence $s_h<_s s_1$, which is a contradiction. Similarly, $\mathfrak{m}_k\notin I_N$, provided that $\omega_n^{l_n}\omega_0^0\in I_N$. In the latter case, by the contrary, if $\mathfrak{m}_0=\mathfrak{m}_k\in I_N$, then, as above, there exists $h$ such that either $s_h\mid  u$ or $s_h\mid  v$. It follows that, either $s_h<_s s_1$ or $s_k<_s  s_h$, which is a contradiction. In conclusion, $\mathfrak{m}_0=\mathfrak{m}_k\notin I_N$.
\end{proof}

\begin{theorem}\label{thm:maxN}
Let $A$ be a special multiserial and locally monomial algebra.  Then $\mathcal{M}_N=\mathcal{S}_N^{\ast}$, for each $N\in\mathcal{D}_{Q,I}$.
\end{theorem}
\begin{proof}
First, we will show that $\mathcal{S}_N^{\ast}\subseteq \mathcal{M}_N$. Let $\mathfrak{m}_i+I_N\in S_N^{\ast}$, with $0<i<k$, 
 and let $\alpha\in(Q_N)_1$ such that $t(\alpha)=s(\mathfrak{m}_i)$. If $\alpha=s_i^0$, then $s_i\mid s_i^0vs_{i+1}^0\cdots s_{i+1}^{\sigma_{i+1}-1}=\alpha\mathfrak{m}_i$, because $s_{i+1}\nmid s_i$. In consequence, $\alpha\mathfrak{m}_i\in I_N$. Now, if $\alpha\neq s_i^0$, then, by def\mbox inition of $\mathcal{S}_N$, we have $\sigma_i>1$ and hence $s_i^0s_i^1\notin I_N$ and since $A$ is a special multiserial algebra, $\alpha\mathfrak{m}_i\in I_N$. A similar reasoning applies if $s(\alpha)=t(\mathfrak{m}_i)$ with $\alpha=s_{i+1}^{\sigma_{i+1}}$ or $\alpha\neq s_{i+1}^{\sigma_{i+1}}$, to obtain $\mathfrak{m}_i\alpha\in I_N$. Furthermore, in a similar way it can be proved that $\mathfrak{m}_0\alpha\in I_N$ and $\alpha\mathfrak{m}_k\in I_N$, for the cases $\omega_n^{l_n}\omega_0^0\in I_N$ or $\omega_n^{l_n}\omega_0^0\notin I_N$ (see Def\mbox inition \ref{def:maxS} $ii)$ and $iii)$). The remaining cases to prove are $\alpha\mathfrak{m}_0\in I_N$ and $\mathfrak{m}_k\alpha\in I_N$, when $\omega_n^{l_n}\omega_0^0\in I_N$. We will show that $\alpha\mathfrak{m}_0\in I_N$. Suppose, by contradiction, that $\alpha\mathfrak{m}_0\notin I_N$. Then, $\alpha\mathfrak{m}_0\mid\omega(N)^{(l)}$, for some $l>0$. By the decomposition of $\omega(N)$ in arrows and the fact that $\mathfrak{m}_0^0=\omega_0^0$, we obtain $\alpha\mathfrak{m}_0^0=\omega_n^{l_n}\omega_0^0$. Nevertheless, the last equality is impossible, since $\alpha\mathfrak{m}_0$ does not contain factors in $I_N$. The case $\mathfrak{m}_k\alpha\in I_N$ is analogous. This completes the proof of the inclusion $\mathcal{S}_N^{\ast}\subseteq \mathcal{M}_N$.

Conversely, for the proof of the inclusion $\mathcal{M}_N\subseteq\mathcal{S}_N^{\ast}$, let $\mathfrak{m}+ I_N\in\mathcal{M}_N$. Since ``$\leq_f$'' is a total order on $(Q_N)_1$, there are three cases:\\
{\bf Case 1:} There is no $i$ such that $s_i^0<_f\mathfrak{m}^0$.\\
\hspace*{\parindent}
In this case, we have that $\mathfrak{m}^0\leq_f s_1^0$. Let $x$, $x'$, $y$, $y'$ be paths in $Q_N$ such that $\omega(N)=x\mathfrak{m}^0x'$ and $\mathfrak{m}^0x'=ys_1^0y'$. We obtain that $\mathfrak{m}\in E_{ys_1^0y'\omega(N)^{(\eta-1)}}$ and $s_1\in E_{s_1^0y'\omega(N)^{(\eta-1)}}$. Since $s_1\nmid\mathfrak{m}$, it follows that $\mathfrak{m}\in E_{y\theta}$, with $\theta=s_1^0\cdots s_1^{\sigma_1-1}$, and hence $\mathfrak{m}\mid xys_1^0\cdots s_1^{\sigma_1-1}$. Notice that $xy$ coincides with the path $u$ in Def\mbox inition \ref{def:maxS} $ii)$ and $iii)$ and since $us_1^0\cdots s_1^{\sigma_1-1}\mid\mathfrak{m}_0$, we get $\mathfrak{m}\mid\mathfrak{m}_0$. In conclusion, $\mathfrak{m}=\mathfrak{m}_0$.\\
{\bf Case 2:} There is no $i$ such that $\mathfrak{m}^0\leq_f s_{i}^0$.\\
\hspace*{\parindent}
In this case, we have that $s_k^0<_f\mathfrak{m}^0$. Let $x$, $y$, $z$ be paths in $Q_N$ such that $\omega(N)=xs_k^0y\mathfrak{m}^0z$. If $\omega_n^{l_n}\omega_0^0\in I_N$, then $\mathfrak{m}\mid\omega(N)$, which implies that $\mathfrak{m}\mid y\mathfrak{m}^0z=\mathfrak{m}_k$, where $z$ coincides with the path $v$ in Def\mbox inition \ref{def:maxS} $ii)$. If $\omega_n^{l_n}\omega_0^0\notin I_N$, then $\mathfrak{m}\mid vus_1^0\rho\omega(N)^{\eta-2}$, where $\rho\in\mathcal{P}(Q_N)$ is such that $s_1^0\rho\in T_{\omega(N)}$. Thus, $\mathfrak{m}\mid vus_1^0\cdots s_1^{\sigma_1-1}=\mathfrak{m}_k$, since $s_1\nmid\mathfrak{m}$. In conclusion, $\mathfrak{m}=\mathfrak{m}_k$.\\
{\bf Case 3:} There exists $i$ such that $s_i^0<_f\mathfrak{m}^0\leq_f s_{i+1}^0$.\\
\hspace*{\parindent}
Let $x$, $y$, $y'$, $z$, $z'$ be paths in $Q_N$ such that $\omega(N)=xs_i^0y\mathfrak{m}^0y'$ and $\mathfrak{m}^0y'=zs_{i+1}^0z'$. Thus, $\mathfrak{m}\in E_{zs_{i+1}^0z'\omega(N)^{(\eta-1)}}$ and $s_{i+1}\in E_{s_{i+1}^0z'\omega(N)^{(\eta-1)}}$. Since $s_{i+1}\nmid \mathfrak{m}$, it follows that $\mathfrak{m}$ is a subpath of $yzs_{i+1}^0\cdots s_{i+1}^{\sigma_{i+1}-1}=\mathfrak{m}_i$. In conclusion, $\mathfrak{m}=\mathfrak{m}_i$.\\
Therefore we have the desired inclusion $\mathcal{M}_N\subseteq\mathcal{S}_N^{\ast}$, which completes the proof of the theorem.
\end{proof}

\begin{remark}\label{rmk:maxlength}
We denote $\mu_{i}:=l_{\mathfrak{m}_i}$. Then, by Def\mbox inition \ref{def:maxS}, we have that:
\begin{figure}[!htb]
   \begin{minipage}{0.48\textwidth}
\centering
$
  \mathfrak{m}_i^0=\begin{cases}
    \omega_0^0, & \text{if $i=0$ and $\omega_n^{l_n}\omega_0^0\in I_N$}.\\
    s_k^1  & \text{if $i=0$ and $\omega_n^{l_n}\omega_0^0\notin I_N$}.\\
    s_i^1, & \text{if $i\neq0$}.
  \end{cases}
$
   \end{minipage}\hfill
   \begin{minipage}{0.48\textwidth}
 \centering
 $
  \mathfrak{m}_i^{\mu_{i}}=\begin{cases}
    \omega_n^{l_n}, & \text{if $i=k$ and $\omega_n^{l_n}\omega_0^0\in I_N$}.\\
   s_{1}^{\sigma_{1}-1}, & \text{if $i=k$ and $\omega_n^{l_n}\omega_0^0\notin I_N$}.\\
   s_{i+1}^{\sigma_{i+1}-1}, & \text{if $i\neq k$}.
  \end{cases}
$
\end{minipage}
\end{figure}

\end{remark}

\begin{lemma}\label{lem:specialcase}
Let $A$ be a special multiserial and locally monomial algebra and let $N\in\mathcal{D}_{Q,I}$. If $|\mathcal{S}_N|=1$ and $\sigma_1>1$, then $A_N$ is a UMP algebra if and only if $\omega_n^{l_n}\omega_0^0\notin I_N$.
\end{lemma}
\begin{proof}
Suppose that $|\mathcal{S}_N|=1$ and $\sigma_1>1$. If $\omega_n^{l_n}\omega_0^0\notin I_N$, then $\mathcal{M}_N=\{s_1^1\cdots s_1^{\sigma_1}qps_1^0\cdots s_1^{\sigma_1-1}+I_N\}$, where $p$ and $q$ are paths in $\mathcal{P}(Q_N)$ such that $ps_1^0\in E_{\omega(N)}$ and $s_1^{\sigma_1}q\in T_{\omega(N)}$. Thus, $A_N$ is a UMP algebra. Conversely, if $\omega_n^{l_n}\omega_0^0\in I_N$, then $\mathfrak{m}_0+ I_N$ and $\mathfrak{m}_1+I_N$ are two distinct non-disjoint maximal paths by Def\mbox inition \ref{def:maxS} $ii)$. Consequently, $A_N$ is not a UMP algebra.
\end{proof}

\begin{lemma}\label{lem:maximal}
Let $A$ be a special multiserial and locally monomial algebra and let $N\in\mathcal{D}_{Q,I}$. Let $\mathfrak{m}+I_N,\, \mathfrak{m}'+I_N$ be two distinct non-disjoint paths in $\mathcal{M}_N$. Then $\mathfrak{m}$ and $\mathfrak{m}'$ are either of the form $\mathfrak{m}=x\rho,\, \mathfrak{m}'=\rho y$ or $\mathfrak{m}=\rho x,\, \mathfrak{m}'= y\rho$, for some non-trivial $x,y,\rho\in\mathcal{P}(Q_N)$.
\end{lemma}
\begin{proof}
Let $\mathfrak{m}+I_N, \mathfrak{m}'+I_N\in \mathcal{M}_N$ be two maximal non-disjoint paths. Since $A$ is a locally monomial algebra there exists a common arrow to $\mathfrak{m}$ and $\mathfrak{m}'$. Let $\alpha$ be this common arrow and consider $\rho$ as the greatest path common to $\mathfrak{m}$ and $\mathfrak{m}'$ containing the arrow $\alpha$. Let $x$, $x'$, $y$, $y'$ be paths in $Q_N$ such that $\mathfrak{m}=x\rho x'$ and $\mathfrak{m}'=y\rho y'$. If $x$ and $y$ are non-trivial, then $x^{l_x}\neq y^{l_y}$ by the choice of $\rho$. Now, since $\mathfrak{m}\notin I_N$ and $\mathfrak{m}'\notin I_N$, it follows that $x^{l_x}\rho^0\notin I_N$ and $y^{l_y}\rho^0\notin I_N$, which is a contradiction with the fact that $A$ is a special multiserial algebra. Hence, $x$ is trivial or $y$ is trivial. Analogously, $x'$ is trivial or $y'$ is trivial. In addition, observe that neither $x$ and $x'$ nor $y$ and $y'$ can be both trivial, by the maximality of $\mathfrak{m}$ and $\mathfrak{m}'$ and the hypothesis that they are dif\mbox ferent. In conclusion, either $\mathfrak{m}=\rho x'$ and $\mathfrak{m}'=y \rho$, with $x'$ and $y$ non-trivial paths or $\mathfrak{m}=x\rho$ and $\mathfrak{m}'=\rho y'$, with $x$ and $y'$ non-trivial paths.
\end{proof}

\begin{theorem}\label{thm:umpchar1}
Let $A$ be a special multiserial and locally monomial algebra and let $N\in\mathcal{D}_{Q,I}$. The algebra $A_N$ is a UMP algebra if and only if it satisf\mbox ies at least one of the following conditions:
\begin{enumerate}
\item[i)] $\sigma_i=1$ for all $i=1,\ldots,k$.
\item[ii)] $k=1$ and $\omega_n^{l_n}\omega_0^0\notin I_N$,
\end{enumerate}
where $\mathcal{S}_N=\{s_1,\ldots, s_k\}$ and $\sigma_i=l_{s_i}$, as before.
\end{theorem}
\begin{proof}
Suppose that $A_N$ is not a UMP algebra. We will prove that $\sigma_h>1$, for some $1\leq h\leq k$, and either $k>1$ or $\omega_n^{l_n}\omega_0^0\in I_N$. Since $A_N$ is not a UMP algebra, there exist distinct and non-disjoint maximal paths $\mathfrak{m}_i+I_N$ and $\mathfrak{m}_j+I_N$, with $i<j$. In particular, $i\neq k$ and $j\neq 0$. By Lemma \ref{lem:maximal}, $\mathfrak{m}_i$ and $\mathfrak{m}_j$ are either of the form $\mathfrak{m}_i=x\rho$, $\mathfrak{m}_j=\rho y$ or $\mathfrak{m}_i=\rho x$, $\mathfrak{m}_j=y\rho$, for some non-trivial $x,y,\rho\in\mathcal{P}(Q_N)$. In the f\mbox irst case, if $\mathfrak{m}_i=x\rho$ and $\mathfrak{m}_j=\rho y$, then by Remark \ref{rmk:maxlength} we have $\rho^{l_{\rho}}=\mathfrak{m}_i^{\mu_i}=s_{i+1}^{\sigma_{i+1}-1}$. Since $\mathfrak{m}_j\mid\omega(N)^{(\eta)}$ and $s_{i+1}\mid\omega(N)^{(\eta)}$, we obtain that $y^0=s_{i+1}^{\sigma_{i+1}}$, by the uniqueness of the factorization. Thus, $s_{i+1}^{\sigma_{i+1}-1}s_{i+1}^{\sigma_{i+1}}\mid\mathfrak{m}_j$ and we have that $\sigma_{h}>1$, with $h=i+1$. In the second case, if $\mathfrak{m}_i=\rho x$ and $\mathfrak{m}_j=y\rho$, then we claim that either $\rho^{l_{\rho}}=\mathfrak{m}_j^{\mu_j}=s_{j+1}^{\sigma_{j+1}-1}$ for $j<k$ or $j=k$ and  $\rho^{l_{\rho}}=\mathfrak{m}_k^{\mu_k}=s_{1}^{\sigma_{1}-1}$. Indeed, if $j<k$ then $\mathfrak{m}_j^{\mu_j}=s_{j+1}^{\sigma_{j+1}-1}$, by Remark \ref{rmk:maxlength}. Now, if $j=k$ then $\omega_n^{l_n}\omega_0^0\notin I_N$. Otherwise, if $\omega_n^{l_n}\omega_0^0\in I_N$, we have that $\mathfrak{m}_k^{\mu_k}=\omega_n^{l_n}$. This implies that, $x^0=\omega_0^0$, because $\omega_n^{l_n}x^0=\mathfrak{m}_k^{\mu_k}x^0=\rho^{l_{\rho}}x^0\mid \mathfrak{m}_i$ and the uniqueness of the factorization. Nevertheless, this is impossible because $\mathfrak{m}_i$ contains no paths in $I_N$. Thus, $\rho^{l_{\rho}}=\mathfrak{m}_k^{\mu_k}=s_{1}^{\sigma_{1}-1}$, by Remark \ref{rmk:maxlength}. In a similar manner as in the f\mbox irst reasoning above, under the hypotheses $\mathfrak{m}_i=\rho x$ and $\mathfrak{m}_j=y\rho$, we obtain that either $x^0=s_{j+1}^{\sigma_{j+1}}$ for $j<k$ or $x^0=s_1^{\sigma_1}$ for $j=k$. Consequently, $\sigma_{h}>1$, where either $h=j+1$ or $h=1$.

Now, notice that if $\omega_n^{l_n}\omega_0^0\notin I_N$, it follows from Def\mbox inition \ref{def:maxS} $iii)$ that $k=|\mathcal{M}_N|>1$, which proves that $k>1$ or $\omega_n^{l_n}\omega_0^0\in I_N$.

Conversely, suppose that $k>1$ or $\omega_n^{l_n}\omega_0^0\in I_N$ and $\sigma_i>1$, for some $i=1,\ldots, k$. Therefore, $|\mathcal{M}_N|\geq k>1$ or $|\mathcal{M}_N|=k+1>1$. By Remark \ref{rmk:maxlength}, we obtain that $\mathfrak{m}_i^0=s_i^1$ and, by Def\mbox inition \ref{def:maxS}, it follows that $s_i^0\cdots s_{i}^{\sigma_i}\in T_{\mathfrak{m}_{i-1}}$, which provides $s_i^1\mid\mathfrak{m}_{i-1}$. Thus, $\mathfrak{m}_{i-1}+I_N$ and $\mathfrak{m}_i+I_N$ are non-disjoint maximal paths in $A_N$, i.e., $A_N$ is not a UMP algebra, which completes the proof.
\end{proof}

For the classif\mbox ication of UMP special multiserial locally monomial algebras, we need the following def\mbox inition.

\begin{definition}\label{dwrel}
Let $A=\Bbbk Q/I$ be a special multiserial and locally monomial algebra and let $r$ be a zero relation in $I$ of length greater than two. We say that $r$ is a \emph{$\omega$-relation} if $r$ is the unique element of $I$ that is a subpath of $r$ (i.e. there is no a proper subpath of $r$ in $I$).
\end{definition}

\begin{remark}\label{romrelnw}
\begin{enumerate}[$i)$]
\item Assume that $r$ is a $\omega$-relation. Using Lemma \ref{lem:components}, we can consider the component $N(r)$. We have that $r\in \mathcal{T}_{N(r)}$ and, in particular, $r$ is a subpath of a power of $\omega(N(r))$. In fact, from Def\mbox inition \ref{dwrel}, for all the arrows $\alpha,\beta\in Q_1$ such that $t(\omega_{\alpha})=s(\omega_{\beta})$, the path $\omega_{\alpha}^{l_{\alpha}}\omega_{\beta}^0\in I$ can not be in $I$ and at the same time be a subpath of $r$. Hence, by \eqref{eq:def2T}, we obtain that $r\in \mathcal{T}_{N(r)}$. By \eqref{eq:defT}, it follows that $r$ is a subpath of a power of $\omega(N(r))$.
\item Let $N\in \mathcal{D}_{Q,I}$. The set of $\omega$-relations in $\mathcal{P}(Q_N)$ conicides with the set of relations of $\mathcal{S}_N$ of length greater than two. Let's verify this claim as follows. First, let $r$ be a $\omega$-relation in $\mathcal{P}(Q_N)$. Since $r\in I\cap \mathcal{P}(Q_N)=I_N$ and $A_N$ is monomial, it follows that there exists a subpath $r'$ of $r$ in $R_N$.
Since $r'\in R_N\subseteq I_N\subseteq I$, we have that $r'=r$, and therefore, $r\in R_N$. Even more, we have that $r\notin \Omega_N$ because the elements of $\Omega_N$ are paths in $I$ of length two, and then $r\in \mathcal{S}_N$. This proves that $\{\omega\mbox{- relations}\}\cap \mathcal{P}(Q_N)\subseteq \mathcal{S}_N$.\\
For the converse, suppose that $r$ is a relation in $\mathcal{S}_N$ with length greater than two. Since $\mathcal{S}_N\subseteq I$, we get $r\in I$. Now, suppose that $r'$ is a subpath of $r$ in $I$. Then, $r'\in \mathcal{P}(Q_N)$ and hence $r'\in I_N$. Since $r\in \mathcal{S}_N\subseteq R_N$ and $R_N$ is minimal, we get $r=r'$. Thus, $r$ is a $\omega$-relation.
\item Note that every zero relation of length greater than two in $R$ is an $\omega$-relation. Moreover, if $A$ is monomial, then $\{\omega-\mbox{relations}\}\subseteq \cup_{N\in \mathcal{D}} \mathcal{S}_N\subseteq \cup_{N\in \mathcal{D}} R_N \subseteq R \subseteq I$. In particular, if $A$ is monomial, then a zero relation $r$ in $I$ is a $\omega$-relation if and only if $r$ is a zero relation of $R$ of length greater than two.
\item Let $J$ be a f\mbox inite subset of the set indexing the relations of $R$. Suppose that every nonzero relation \(r = \sum_{j \in J} \lambda_j p_j \in R\), satisf\mbox ies the condition that for any pair of arrows \( \beta\),  \(\gamma\) and any \( j \in J \), both \(\beta p_j^0\) and \(p_j^{l_j} \gamma\) are in the ideal $I$. Under these conditions, we obtain that:
\begin{enumerate}
    \item If \( A = \Bbbk Q / I \) is special multiserial algebra, then for every path $p_j$ that appears in \( r \), we have \(\omega\big( N(p_j) \big) = p_j\), where \( N(p_j) \) is the weakly connected component containing \( p_j \). This result, in particular, implies that \( A \) is locally monomial.
    \item The products \( r \gamma = \sum_{j \in J} \lambda_j \, \hat{p}_j \big( p_j^{l_j} \gamma \big)\), and \(\beta r = \sum_{j \in J} \big( \beta p_j^0 \big) \lambda_j \tilde{p}_j\), can each be expressed as a linear combination of quadratic zero relations. This implies that any element \( p \in I \), can be written in the form:
 \[
 p = \sum_{\substack{r \ \text{zero relation in} \ R}} \lambda_r \, p_r \, r \, p'_r  
 \;+\;  
 \sum_{\substack{r \ \text{nonzero relation in} \ R}} \lambda_r \, r.
 \]
 In particular, a zero relation \( r \in I \) is an \(\omega\)-relation if and only if \( r \) is a zero relation in \( R \) of length greater than two.
\end{enumerate}
\end{enumerate}
\end{remark}

\begin{theorem}\label{thm:main}
Let $A=\Bbbk Q/I$ be a special multiserial and locally monomial algebra. Then, $A$ is a UMP algebra if and only if for every $\omega$-relation $r$, the path $\omega(N(r))$ is cyclic and $r$ is the unique relation in $R_{N(r)}$ that is a subpath of some power of $\omega(N(r))$, i.e.
\begin{equation}
R_{N(r)}\cap\{\text{subpaths of}\,\, \omega(N(r))^{(p)}:\,p\in\mathbb{Z}^+\}=\{r\}.
\end{equation}
\end{theorem}

\begin{proof}

First, suppose that $A$ is a UMP algebra. Let $r$ be a $\omega$-relation. By the parts $i)$ and $iii)$ of Remark \ref{romrelnw}, $r$ is a relation in $\mathcal{S}_{N(r)}$ and is a subpath of a power of $\omega(N(r))$. Because $l_r>1$, from Theorem \ref{thm:umpchar1} we obtain that $\mathcal{S}_{N(r)}=\{r\}$ and $\omega_n^{l_n}\omega_0^0\notin I_{N(r)}$. On the other hand, since $\omega_i^{l_i}\omega_{i+1}^0\neq 0$ for all $i\in \{0,\ldots,n\}$, Theorem \ref{thm:omegaN} implies that there are no subpaths of any power of $\omega(N(r))$ in $\Omega_{N(r)}$. Hence, by \eqref{eq:defT} and \eqref{eq:defT2}, we have that $R_{N(r)}\cap \mathcal{T}_{N(r)}\subseteq S_{N(r)}$. Therefore, the unique relation of $R_{N(r)}$ that is a subpath of some power of $\omega(N(r))$ is $r$.

Conversely, suppose that $R_{N(r)}\cap\{\text{subpaths of}\,\,\omega(N(r))^{(p)}:p\in\mathbb{Z}^+\}=\{r\}$ for every $\omega$-relation $r$. Let $N\in\mathcal{D}_{Q,I}$. Assume that $A$ is not UMP. By Theorem \ref{thm:umpchar1}, there is a relation $r$ in $\mathcal{S}_N$ of length greater than two, and moreover, it holds that $|\mathcal{S}_N|>1$ or $\omega_n^{l_n}\omega_0^0\in I$. Since $\omega_n^{l_n}\omega_0^0$ and the elements of $\mathcal{S}_N$ are subpaths of powers of $\omega(N)$, we obtain a subpath of a power of $\omega(N)$ in $R_N$ dif\mbox ferent from $r$, which is a contradiction. Thus, $A_N$ is UMP. In consequence, by Corollary \ref{cor:local}, $A$ is UMP.

\end{proof}

A direct application of Theorem \ref{thm:main} is in the monomial case, where its statement and interpretation are more manageable.

\begin{corollary}\label{cor:main}
Let $A=\Bbbk Q/I$ be a special multiserial and monomial algebra and let $R$ be a minimal set of relations such that $I=\langle R \rangle$. Then, $A$ is a UMP algebra if and only if for any zero relation $r\in R$ with length greater than two, there exists a path $u$,  such that $t(u)=s(u)$ and 
\begin{equation*}
R\cap\{\text{subpaths of}\,\, u^{(p)}:\,p\in\mathbb{Z}^+\}=\{r\}. 
\end{equation*}
\end{corollary}

\begin{proof}
Is a direct consequence of Theorem \ref{thm:main} and Remark \ref{romrelnw}.\,$iii)$.
\end{proof}

Under suitable conditions, Corollary \ref{cor:main} can be extended to a broader class of algebras, as shown in the result below.

\begin{corollary}\label{cor:conesp}
Let $A = \Bbbk Q/I$ be a special multiserial algebra, with \( R \) a minimal set of relations such that \( I = \langle R \rangle \). Suppose that every nonzero relation \(r = \sum_{j \in J} \lambda_j p_j \in R\), satisf\mbox ies the condition that for any pair of arrows \( \beta\),  \(\gamma\) and any \( j \in J \), both \(\beta p_j^0\) and \(p_j^{l_j} \gamma\) are in the ideal $I$. Then, 
\( A \) is a UMP algebra if and only if, for any zero relation \( r \in R \) with length greater than two, there exists a path \( u \) such that \( t(u) = s(u) \) and 
\begin{equation*}
    R \cap \{\text{subpaths of } u^{(p)} : p \in \mathbb{Z}^+ \} = \{r\}. 
\end{equation*}
\end{corollary}

\begin{proof}
Is a direct application of Theorem \ref{thm:main} and Remark \ref{romrelnw}.\,$iv)$.
\end{proof}

The following proposition provides tools for determining whether a specif\mbox ic algebra has the UMP property in concrete cases (see also Remark \ref{rmk:practice}).

\begin{proposition}\label{prop:practice}
Let \( A = \Bbbk Q / I \) be a special multiserial, locally monomial algebra, with $R$ being a minimal generating set of relations for $I$. Suppose there is a non-zero relation $\rho=\sum_{j \in J} \lambda_j p_j \in R$ such that for some $i\in J$, $p_i$ is not a power of a loop. If there exists an arrow $\beta$ satisfying either $\beta p_i^{0} \notin I$ or $p_i^{l_i} \beta \notin I$, then $A$ is not UMP.
\end{proposition}

\begin{proof}
Without loss of generality, we can assume that $1\in J$, $i = 1$ and $\lambda_1 = 1$. Suppose there exists an arrow, such that \(\beta p_1^{0} \notin I\). Since \(A\) is a special multiserial algebra, we have that \(\beta p_j^{0} \in I\), for all \(j\neq 1\). Consequently, \(\beta p_1 \in I\), and because \(\rho \in R\), we get \(\beta p_1 \in R_{N(p_1)}\).\\
Since \(\rho\) is not monomial, it follows that \(\omega_\beta \neq \omega_{p_1^{0}}\), and both are vertices of \(N(p_1)\), joined by an arrow 
\(\omega_\beta \to \omega_{p_1^{0}}\). Furthermore, for any arrow \(\alpha \neq \beta\) we have \(\alpha p_1^{0} \in I\), which means there is no arrow  
\(\omega_\alpha \to \omega_{p_1^{0}}\) in \(N(p_1)\). Consequently, there exists a maximal path \(\mathfrak{m}\) whose f\mbox irst arrow is \(p_1^{0}\).\\ 
We will now prove that the last arrow of $\mathfrak{m}$ is $p_1^{l_1}$. If $p_1^{l_1}\gamma \in I$, it follows that $\omega_\gamma \notin N(p_1)$, and thus
$\mathfrak{m}$ cannot contain the arrow $\gamma$. On the other hand, if there exists an arrow \(\gamma\) such that \(p_1^{l_1}\gamma \notin I\), we can conclude, as shown previously, that \(p_1\gamma \in R_{N(p_1)}\). In either case, the last arrow of \(\mathfrak{m}\) must be \(p_1^{l_1}\), which allows us to conclude that \(\mathfrak{m} = p_1\). Def\mbox initively, we establish that if \(\beta p_1^{0} \notin I\), then \(\mathfrak{m} = p_1\).\\
Since \(A\) is locally monomial, its component \(A_{N(p_1)}\) is monomial. Now suppose \(\mathfrak{m}'\) is a maximal path whose last arrow is $p_1^{l_1-1}$. If \(\mathfrak{m}'+ I_{N(p_1)} = \mathfrak{m} + I_{N(p_1)}\), it follows that \(\mathfrak{m}' = \mathfrak{m}\). This would imply \(p_1^{j} = p_1^{j+1}\) for all \(0 \leq j \leq l_i - 1\), which contradicts our initial hypothesis that $p_1$ is not a power of a loop. This contradiction establishes that \(\mathfrak{m}\neq \mathfrak{m}'\), proving that \(A_{N(p_1)}\) is not UMP algebra. We thus conclude that \(A\) itself is not UMP, by Corollary \ref{cor:local}. The proof for the case where \(p_1^{l_1} \beta \notin I\) is analogous.
\end{proof}

\begin{remark}\label{rmk:practice}
Below, we illustrate how to determine if the UMP property holds by analyzing the minimal set of relations in practical cases. For a special multiserial and locally monomial path algebra $A=\Bbbk Q/I$, where $R$ is a minimal generating set of relations for $I$,  we proceed with the following steps:
\begin{itemize}
\item {\bf Step 1:} If $R$ also contains nonzero relations, and each of these satisf\mbox ies the condition in Remark~\ref{romrelnw}.\,iv), then Corollary~\ref{cor:conesp} applies.
\item{\bf Step 2:} If there is a nonzero relation in $R$ that does \emph{not} satisfy Remark~\ref{romrelnw}.\,iv), choose a summand path $p$ from that relation and takes an arrow $\beta\in Q_1$ such that $\beta p^0\notin I$ or $p^{l_p}\beta\notin I$.
\begin{enumerate}
\item If $p$ is not a power of a loop, Proposition~\ref{prop:practice} implies that $A$ is \emph{not} UMP.
\item If $p=c^{m}$ is a power of a loop $c$, the special multiserial property ensures that the maximal path containing $p$ is $c^{m+1}$. 
\end{enumerate}We then repeat the check for all remaining nonzero relations. If all of these relations satisfy Remark~\ref{romrelnw}.\,iv), we can apply Corollary~\ref{cor:conesp}.
\end{itemize}

\end{remark}

\begin{exam}\label{eowrn}

\begin{enumerate}[$i)$]

Consider the algebra $A$ in Example \ref{exgdr}.$i)$. For the $\omega$-relation $r=abc$ in Example \ref{exlmsmtos}.$i)$,  we have that $\omega(N(r))=dabce$ which is not cyclic. Then, $A$ is not UMP.

\item Consider the algebra $A$ in Example \ref{exgdr}.$ii)$. We obtain that the unique $\omega$-relations of $A$ are $r_1:=a^3$ and $r_2:=b^3$ for which the paths $\omega(N(r_1))=\omega_a=a$ and $\omega(N(r_2))=\omega_b=b$ are cyclic. Moreover, we have that $R_{N(r_1)}=\{a^3\}$ and $R_{N(r_2)}=\{b^3\}$. From Theorem \ref{thm:main}, we have that $A$ is UMP.

\item Consider the algebra $A$ in Example \ref{exgdr}.$iii)$. For the $\omega$-relation $r=abcda$ in Example \ref{exlmsmtos}.$ii)$, we have that $\omega(N(r))=abcd$ which is a cyclic path, but the $\omega$-relation $dabcd$ is another relation which is a subpath of a power of $\omega(N(r))$. Then, $A$ is not UMP.

\end{enumerate}

\end{exam}

\section{Final comments}
\subsection{Brauer graph algebras}
In this subsection, we apply the tools developed in this paper for the case of symmetric special biserial algebras. It is known that these algebras are of tame representation type and they can be obtained as Brauer graph algebras. In the following theorems we describe the form of every weakly connected component of $G_{Q,I}$ for a symmetric special biserial algebra $A=\Bbbk Q/I$. Moreover, we determine which of these algebras are also UMP algebras. This characterization is given in terms of their bound quiver and their associated Brauer graph.

To this end, we f\mbox irst present some preliminary def\mbox initions on Brauer graph algebras. See \cite{D} and \cite{S} for more details.

A Brauer graph is a quadruple $\mathcal{G}=(V,E,m,\mathfrak{o})$ consisting of the following data:
\begin{enumerate}
    \item The pair $(V,E)$ is a f\mbox inite and connected graph (loops and multiple edges are allowed).
For each $v\in V$, we def\mbox ine the \textit{valency} of $v$, denoted by $val(v)$, as the number of \textit{half-edges} incident to $v$.
    \item A function $m:V\rightarrow \mathbb{Z}^+$ called \textit{multiplicity function}. We say that a vertex $v\in V$ is \textit{truncated} if $val(v)m(v)=1$. We denote by $V^{\ast}$ the set of non-truncated vertices of $\mathcal{G}$.
    \item A function $\mathfrak{o}$, called \textit{orientation}, that assigns to each non-truncated vertex $v\in V^{\ast}$ a cyclic order on all the edges incident to $v$ such that, if $val(v)=1$, then $\mathfrak{o}(v)$ is given by $i<i$, where $i$ is the unique edge incident to $v$. Thus, for every non-truncated vertex $v\in V$, the orientation $\mathfrak{o}(v)$ can be written as $i_0<i_1<\cdots<i_{val(v)-1}<i_0$, where $i_0,i_1,\cdots,i_{val(v)-1}$ are all the edges incident to $v$. In this case, we say that $i_{k+1}$ is a \textit{successor} of $i_k$ if $0\leq k<val(v)-1$ and $i_0$ is the successor of $i_{val(v)-1}$.
\end{enumerate}

We represent any Brauer graph $\mathcal{G}$ into an oriented plane such that $\mathfrak{o}$ is given by a counterclockwise orientation. For each $v\in V^{\ast}$ and each edge $i$ incident to $v$, the cyclic order $\mathfrak{o}$ has the form $i_0:=i<i_1<\cdots<i_{val(v)-1}<i_0$. We call to $i_0,\ldots,i_{val(v)-1}$ a \textit{successor sequence} of $v$ starting at $i$. An edge can be appear twice in the cyclic order of a vertex $v$ and hence, in general, there is not a unique successor sequence of $v$ starting at $i$. For this reason, we must consider half-edges. To f\mbox ix notation, we write $\widehat{i}$ and $\widetilde{i}$ to denote the half-edges corresponding to the edge $i$. If $i\in E$ is not a loop, then we just write $\widehat{i}=\widetilde{i}=i$. Observe that every pair consisting of a vertex $v$ and a half-edge $\widehat{i}$ of $i\in E$ incident to $v$ determines a unique successor sequence of $v$ starting at $i$. We denote by $\zeta(v,\widehat{i})$ to such successor sequence.

We def\mbox ine the Brauer graph algebra associated to $\mathcal{G}$ as the bound quiver algebra $A_{\mathcal{G}}:=\Bbbk Q_{\mathcal{G}}/I_{\mathcal{G}}$, where the quiver $Q_{\mathcal{G}}=(Q_0,Q_1,s,t)$ is given as follows:
\begin{itemize}
    \item $Q_0=E$.
    \item Let $i,j\in Q_0$. We put an arrow $\alpha\in Q_1$ from $i$ to $j$ provided that $i$ and $j$ are incident to a common non-truncated vertex $v\in V$ and $j$ is a successor of $i$ in the cyclic order of $v$. Notice that if $v$ is a non-truncated vertex and $i$ is an incident edge to $v$, then, for each half-edge $\widehat{i}$ of $i$, the successor sequence $\zeta(v,\widehat{i})$ induces a cyclic path (starting at $i$) $\alpha_0\cdots \alpha_{val(v)-1}$ in $\mathcal{P}(Q_{\mathcal{G}})$, with $\alpha_k\in Q_1$ for $0\leq k\leq val(v)-1$. This path is called a \textit{special $v$-cycle} and is denoted by $A_{v,\widehat{i}}$.
    \item The ideal $I$ of the path algebra $\Bbbk Q_{\mathcal{G}}$ is generated by the relations of type I, II and III def\mbox ined as follows.
    \begin{enumerate}
        \item[{\bf Type I:}] $A_{v,\widehat{i}}^{(m(v))}-A_{v',\widetilde{i}}^{(m(v'))}$, where $v$ and $v'$ are non-truncated vertices and 
        $i$ is an edge linking $v$ with $v'$ and the half-edges $\widehat{i}$ and $\widetilde{i}$ are incident to $v$ and $v'$, respectively.
        \item[{\bf Type II:}] $A_{v,\widehat{i}}^{(m(v))}A_{v,\widehat{i}}^0$, where $v$ is a non-truncated vertex, $i$ is an edge incident to $v$ and $\widehat{i}$ is a half-edge of $i$. Also, $A_{v,\widehat{i}}^0$ denotes the f\mbox irst arrow of $A_{v,\widehat{i}}$.
        \item[{\bf Type III:}] $\alpha \beta$, where $\alpha, \beta \in Q_1$, $\alpha \beta$ is not a subpath of any special cycle, except if $\alpha=\beta$ is a loop corresponding to a non-truncated vertex $v\in V$ such that $val(v)=1$.
    \end{enumerate}
\end{itemize}

\begin{theorem}\label{thm:Bralg}
Let $\mathcal{G}$ a Brauer graph, $A_{\mathcal{G}}=\Bbbk Q_{\mathcal{G}}/I_{\mathcal{G}}$ its corresponding Brauer graph algebra and $G_{Q_{\mathcal{G}},I_{\mathcal{G}}}$ the ramif\mbox ications graph of $A_{\mathcal{G}}$. Then, the weakly connected components of $G_{ Q_{\mathcal{G}},I_{\mathcal{G}}}$ are in bijection with the non-truncated vertices of $\mathcal{G}$. If $v$ is the non-truncated vertex corresponding to a component $N\in \mathcal{D}_{Q_{\mathcal{G}},I_{\mathcal{G}}}$, then $\omega(N)=A_{v,\widehat{i}}$ for some $i\in E$ incident to $v$. Moreover, each one of the weakly connected components of $G_{ Q_{\mathcal{G}},I_{\mathcal{G}}}$ is just a vertex or is of the form
$$\xymatrix{
\bullet \ar[r] & \bullet  & \cdots & \ar[r]
& \bullet \ar@/^2pc/[llll]}$$
\vspace{3mm}
\end{theorem}

\begin{proof} We can assume that $V^{\ast}\neq \emptyset$. First, we will prove that if $\alpha\in Q_1$, then there exists $v\in V^{\ast}$ such that $\omega(N(\alpha))$ or a rotation of $\omega(N(\alpha))$ equals $A_{v,\widehat{s(\alpha)}}$. By the def\mbox inition of $A_{\mathcal{G}}$, there is a non-truncated vertex $v$ of $\mathcal{G}$ for which $s(\alpha),t(\alpha)\in E$ and $t(\alpha)$ is successor of $s(\alpha)$. Hence, we have a special $v$-cycle at $s(\alpha)$ of the form $A_{v,\widehat{s(\alpha)}}=\alpha_0\cdots \alpha_{val(v)-1}$ such that $\alpha_0=\alpha$, for some half-edge $\widehat{s(\alpha)}$ of $s(\alpha)$. Now, note that the length of the path $(A_{v,\widehat{s(\alpha)}})^{(m(v))}$ is given by the positive integer $m(v)val(v)$ which is greater than $1$. Therefore, since $A_{v,\widehat{s(\alpha)}}^{(m(v))}+I_{\mathcal{G}}$ is a maximal path, all the subpaths of length two of any power of $A_{v,\widehat{s(\alpha)}}$ do not belong to $I_{\mathcal{G}}$. Thus, by the def\mbox inition of the relations of type III, these are all the paths of length two in $\mathcal{P}(Q_{\mathcal{G}})$ that do not belong to $I_{\mathcal{G}}$. This implies that $A_{v,\widehat{s(\alpha)}}$ is a path in $A_{N(\alpha)}$ that contains all the arrows of $Q_{N(\alpha)}$ and, consequently, either $A_{v,\widehat{s(\alpha)}}$ coincides with $\omega(N(\alpha))$ or is a rotation of $\omega(N(\alpha))$.

We def\mbox ine the map $\psi:\mathcal{D}_{Q_{\mathcal{G}},I_{\mathcal{G}}}\rightarrow V^{\ast}$ as follows: for each $N\in \mathcal{D}_{Q_{\mathcal{G}},I_{\mathcal{G}}}$, def\mbox ine $\psi(N)$ as the non-truncated vertex of $\mathcal{G}$  for which $t(\omega(N)^0)$ is the successor of $s(\omega(N)^0)$ in its corresponding cyclic order. The map $\psi$ is independent of the choice of $\omega(N)$ because, by the previous reasoning, every rotation of $\omega(N)$ is also a special $\psi(N)$-cycle, and hence its f\mbox irst arrow corresponds to a successor relation of the same successor sequence as for $\omega(N)$.

On the other hand, we def\mbox ine the map $\varphi:V^{\ast}\rightarrow \mathcal{D}_{Q_{\mathcal{G}},I_{\mathcal{G}}}$ in the following fashion. Let $v\in V^{\ast}$, $i\in E$ and $\widehat{i}$ a half-edge of $i$ incident to $v$. Def\mbox ine $\varphi(v)$ as the component in $\mathcal{D}_{Q_{\mathcal{G}},I_{\mathcal{G}}}$ such that $A_{v,\widehat{i}}^{(\mathfrak{m}(v))}\in \mathcal{P}(Q_{\varphi(v)})$. We can do this because $A_{v,i}^{(\mathfrak{m}(v))}$ is a non-zero path in $A$ of lenght $m(v)val(v)>1$ (see Remark \ref{rem:unicomp}.$i)$). The map $\varphi$ is well-def\mbox ined because if instead taking $A_{v,\widehat{i}}$, we choose $A_{v,\widetilde{j}}$ for a half-edge $\widetilde{j}$ of an edge $j\in E$ with $\widetilde{j} \neq \widehat{i}$, then $A_{v,\widetilde{j}}$ is a rotation of $A_{v,\widehat{i}}$ and hence $A_{v,\widehat{i}}^{(\mathfrak{m}(v))}$ and $A_{v,\widetilde{j}}^{(\mathfrak{m}(v))}$ are paths in the quiver of the same weakly connected component.

It is easy to check that $\psi$ and $\varphi$ are inverses each other.

Finally, given $N\in \mathcal{D}_{Q_{\mathcal{G}},I_{\mathcal{G}}}$, since any power of $A_{\psi(N),\widehat{s(\omega(N))}}$ contains no zero relations of length two, it also follows that $\omega(N)^{l_{\omega(N)}}\omega(N)^0 \notin I_{\mathcal{G}}$. Since $A_{\mathcal{G}}$ is a special multiserial algebra, if $N$ is not just a vertex, the last condition and Lemma \ref{lem:compchar} imply that the weakly connected component $N$ is a cycle.

\end{proof}

\begin{remark}
Based on the proof of Theorem \ref{thm:Bralg}, the properties of the vertices in a Brauer graph determine whether a component of its algebra is a single vertex or a cyclic component. A vertex with a loop always results in a cyclic component. For a non-truncated vertex without loops, the associated component is reduced to a single vertex if and only if it is adjacent to at most one other non-truncated vertex. Furthermore, this component is cyclic if and only if the vertex is adjacent to at least two distinct non-truncated vertices. To prove these last cases, we f\mbox irst observe that each non-truncated vertex \(v\) of \(\mathcal{G}\) without loops determines a sequence of successors
 \[
i= i_0 < i_1 < \cdots < i_{{val}(v)-1},
 \]
 which is associated with a special $v$-cycle in \(Q_\mathcal{G}\). If $v$ is the only non-truncated vertex of $\mathcal{G}$, then $Q_{\mathcal{G}}$ is a cycle quiver and hence we have an only vertex in $G_{Q_{\mathcal{G}},I_{\mathcal{G}}}$.\\If \(v\) is adjacent to exactly one non-truncated vertex \(u\), then the cycle \(A_{v,i}\) shares only the vertex $i$ with \(A_{u,i}\). In this case, it is clear that \(\omega_{\alpha_0} = A_{v,i}\).\\
 Now, if a vertex \(v\) is adjacent to at least two non-truncated vertices, say \(j\) of them, then the cycle \(A_{v,i}\) will have \(j\) vertices that also belong to another cycle. \\
 Let \(u\) and \(u'\) be non-truncated vertices. Suppose the subsequence
 \[
 i_{l_u} < i_{l_u+1} < \cdots < i_{l_{u'}}
\]
 of \(i_0 < i_1 < \cdots < i_{{val}(v)-1}\) contains no \(i_l\) corresponding to an edge joining \(v\) to a non-truncated vertex other than \(u\) and \(u'\). In this case, we have a path
 \[
 \omega_{\alpha_{l_u}} = \alpha_{l_u} \cdots \alpha_{l_{u'}-1},
 \]
where \(\alpha_l : i_l \to i_{l+1}\). This construction shows that the weakly connected component is a cycle with $j$ vertices. Now, suppose that the vertex $v$ has a loop, denoted by $i$.  The corresponding vertex $i$ in the quiver $Q_\mathcal{G}$ has two incoming and two outgoing arrows. In general, for any two distinct arrows $\alpha, \beta \in (Q_\mathcal{G})_1$ with $s(\alpha) = i = s(\beta)$, we have $\omega_\alpha \neq \omega_\beta$. Therefore, this component is also a cycle.
\end{remark}

Using the bijection between $\mathcal{D}_{Q_{\mathcal{G}},I_{\mathcal{G}}}$ and $V^*$, we obtain a formula for $\dim_{\Bbbk}((A_{\mathcal{G}})_N)$ for any weakly connected component $N$ of $G_{Q_{\mathcal{G}},I_{\mathcal{G}}}$ in terms of the data given by the Brauer graph $\mathcal{G}$.

\begin{theorem}\label{thm:trunc}
Let $N\in \mathcal{D}_{Q_{\mathcal{G}},I_{\mathcal{G}}}$. Then, $\dim_{\Bbbk}((A_{\mathcal{G}})_N)=val(v)(val(v)m(v)+1)$, where $v$ is the unique non-truncated vertex in $\mathcal{G}$ for which $\omega(N)$ has the form $A_{v,\,\widehat{i}}\,.$
\end{theorem}

\begin{proof}
Let $N\in \mathcal{D}_{Q_{\mathcal{G}},I_{\mathcal{G}}}$ and let $v$ be the non-truncated vertex in $\mathcal{G}$ for which $\omega(N)$ has the form $A_{v,\,\widehat{i}}\,.$ Note that every maximal path in $(A_{\mathcal{G}})_N$ has the form $A_{v,\,\widehat{i}}^{(m(v))}+(I_{\mathcal{G}})_N$ and its representant $A_{v,\widehat{i}}^{(m(v))}$ has length $val(v)m(v)$. Hence, the length of every path in $(A_{\mathcal{G}})_N$ is between $0$ and $val(v)m(v)$.
Moreover, for each integer $l$, with $0\leq l\leq val(v)m(v)$, and for each half-edge $\widehat{i}$ incident to $v$, there exists a unique path of length $l$ in $(A_{\mathcal{G}})_N$, with source $i$, following the cyclic order given by $v$ and starting at $\widehat{i}$. This implies that the number of paths in $(A_{\mathcal{G}})_N$ of length $l$ is $val(v)$. Thus,
$$
\dim_{\Bbbk}((A_{\mathcal{G}})_N)=|\mathcal{P}(Q_N,I_N)|=\sum_{l=0}^{val(v)m(v)}val(v)=val(v)(val(v)m(v)+1).
$$
\end{proof}

\begin{lemma}\label{lem:BGALM}
Let $\mathcal{G}=(V,E,m,\mathfrak{o})$ be a Brauer graph and $A_{\mathcal{G}}=\Bbbk Q_{\mathcal{G}}/I_{\mathcal{G}}$ its corresponding Brauer graph algebra. Then, $A_{\mathcal{G}}$ is locally monomial if and only if $\mathcal{G}$ has no loops.
\end{lemma}

\begin{proof}
First, suppose that there are a vertex $v\in V$ and a loop $i\in E$ such that $i$ is incident to $v$. We will prove that $A_{\mathcal{G}}$ is not locally monomial. Recall that we denote by $\widehat{i}$ and $\widetilde{i}$ the half-edges corresponding to $i$. Then, $A_{v,\widehat{i}}$ and $A_{v,\widetilde{i}}$ are two dif\mbox ferent non-zero paths which are not disjoint. Applying Remark \ref{rem:unicomp}.$i)$, we obtain that $N(A_{v,\widehat{i}})=N(A_{v,\widetilde{i}})$ and hence the commutative relation $A_{v,\widehat{i}}^{m(v)}-A_{v,\widetilde{i}}^{m(v)}\in I$ also belongs to $I_{N(A_{v,\widehat{i}})}$. Therefore, $(A_{\mathcal{G}})_{N(A_{v,\widehat{i}})}$ is not monomial. Thus, $A_{\mathcal{G}}$ is not locally monomial.

For the converse suppose that $\mathcal{G}$ has no loops and let $N\in \mathcal{D}_{Q_{\mathcal{G}},I_{\mathcal{G}}}$. We will prove that $A_N$ is monomial. By Theorem \ref{thm:Bralg}, there is $v\in \mathcal{G}$ such that $\omega(N)=A_{v,i}$ for some edge $i\in E$ incident to $v$. We will see that there is no any relation of type I in $I_N$. If there is such relation, namely, $A_{v,i}^{(m(v))}-A_{w,i}^{(m(w))}$ for some $w\in V$, then $w\neq v$ because $i$ is not a loop, and hence, $A_{v,i}$ and $A_{w,i}$ are disjoint. Now, suppose that $A_{v,i}=\alpha_0,\ldots,\alpha_{val(v)-1}$ and $A_{w,i}=\beta_0,\ldots,\beta_{val(w)-1}$, where $\alpha_0,\ldots,\alpha_{val(v)-1},\beta_0,\ldots,\beta_{val(w)-1}$ are arrows in $Q_1$. Since $\alpha_{val(v)-1}\beta_0\in I$ and $\beta_{val(w)-1}\alpha_0\in I$, there is no any edge connecting $A_{v,i}$ with $A_{w,i}$. Therefore, $N(A_{v,i})\neq N(A_{w,i})$, which implies that $A_{v,i}^{(m(v))}-A_{w,i}^{(m(w))}\notin I_N$. In consequence, $I_N$ is generated by all the paths of the form $A_{v,i}^{(m(v))}A_{v,i}^0$ and $\alpha\beta$, where $i\in E$ is an edge incident to $v$ and $\alpha,\beta\in (Q_N)_1$ are arrows such that $\alpha\beta$ is not a subpath of any special $v$-cycle. Thus, $A_N$ is monomial.

\end{proof}

Despite the previous result, in the following theorem we classify all the UMP Brauer graph algebras (whether or not they are locally monomial).

\begin{theorem}\label{ubg}
Let $\mathcal{G}$ be a Brauer graph and $A_{\mathcal{G}}=\Bbbk Q_{\mathcal{G}}/I_{\mathcal{G}}$ its corresponding Brauer graph algebra. Then, the following statements are equivalent.
\begin{enumerate}
\item $A_{\mathcal{G}}$ is a UMP algebra.

\item The Brauer graph $\mathcal{G}$ has one of the forms

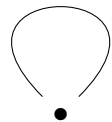
\vspace{3mm}

\begin{figure}[!htb]
   \begin{minipage}{0.48\textwidth}
\centering
 $\,\xymatrix{
 \bullet \ar@{-}[r] & \bullet
 }$
 \caption{}
   \end{minipage}\hfill
   \begin{minipage}{0.48\textwidth}
 \centering
 \begin{tikzpicture}
    \node (bu) {$\bullet$} ;
    \path  (bu)   edge[my loop] node[above]  {} (bu);
\end{tikzpicture}
\end{minipage}
\end{figure}

\item The quiver $Q_{\mathcal{G}}$ and the ideal $I_{\mathcal{G}}$ are given by one of the following four cases
\begin{enumerate}
    \item $Q_0=\{0\}$ and $Q_1=\emptyset$. In this case, $A_{\mathcal{G}}\cong \Bbbk$.
    \item \ \ $Q_{\mathcal{G}}: \xymatrix{0 \ar@(ur,ul)[]_{\alpha} &  \mbox{ and } \ \ I_{\mathcal{G}}=\langle \alpha^{m+1}\rangle}$, with $m\in \mathbb{Z}^+$.

In this case, $A_{\mathcal{G}}\cong \Bbbk[ \,x] \,/\langle x^{m+1} \rangle$.
    \item \ \ $Q_{\mathcal{G}}: \xymatrix{0 \ar@(ru,lu)[]_{\alpha}\ar@(ld,rd)[]_{\beta} &  \mbox{ and } \ \ I_{\mathcal{G}}=\langle \alpha^m-\beta^n, \alpha\beta, \beta\alpha \rangle}$, where $m,n\in \mathbb{Z}_{\geq 2}$.

In this case, $A_{\mathcal{G}}\cong \Bbbk\langle\,x,y\rangle \,/\langle xy,yx,x^m-y^n\rangle$.
    \item\ \ $Q_{\mathcal{G}}: \xymatrix{0 \ar@(ru,lu)[]_{\alpha}\ar@(ld,rd)[]_{\beta} &  \mbox{ and } \ \ I_{\mathcal{G}}=\langle (\alpha\beta)^m-(\beta\alpha)^m, \alpha^2, \beta^2 \rangle}$, where $m\in \mathbb{Z}^+$.

In this case, $A_{\mathcal{G}}\cong \Bbbk\langle\,x,y\rangle \,/\langle x^2,y^2,(xy)^m-(yx)^m\rangle$.
    
\end{enumerate} 
\end{enumerate}
\end{theorem}

\begin{proof} First, assume that $A_{\mathcal{G}}$ is a UMP algebra. Suppose that $v$ is a vertex of $\mathcal{G}$ such that there are two dif\mbox ferent edges $i$ and $j$ incident to $v$, where $j$ is a successor of $i$ in the cyclic order associated to $v$. Hence, there exist half-edges $\widehat{i}$ and $\widehat{j}$ which induce an arrow $\alpha$ from $i$ to $j$ in $Q_1$. In this case, the maximal paths $A_{v,\hat{i}}^{(m(v))}+I$ and $A_{v,\hat{j}}^{(m(v))}+I$ are not disjoint due to the fact that $(A_{v,\hat{i}}^{(m(v))})^0=\alpha=(A_{v,\hat{j}}^{(m(v))})^l$, where $l=l_{A_{v,\hat{j}}^{(m(v))}}$. Since $A_{\mathcal{G}}$ is a UMP algebra, it follows that $A_{v,\hat{i}}^{(m(v))}+I=A_{v,\hat{j}}^{(m(v))}+I$. Nevertheless, by def\mbox inition of the relations of type I, it only occurs when $i=j$, which is a contradiction. Thus, for any vertex $v\in V$ there is only one edge incident to $v$ and hence $\mathcal{G}$ is one of the Brauer graphs in $(2)$.

Conversely, if $\mathcal{G}$ is one of the Brauer graphs in $(2)$, then a direct calculation shows that $A_{\mathcal{G}}$ is one of the bound quiver algebras given in $(3)$. Finally, if $A_{\mathcal{G}}$ is one of the algebras given in $(3)$, then, by Lemma \ref{lem:BGALM} and Theorem \ref{thm:main}, the algebra $A_{\mathcal{G}}$ is a UMP algebra in cases $(a)$, $(b)$ and $(c)$. For the f\mbox inal case $(d)$, a straightforward computation allows us to conclude that $A_{\mathcal{G}}$ is also a UMP algebra, which completes the proof.
\end{proof}

\begin{remark}
\begin{enumerate}
\item An interesting subclass of the Brauer graph algebras are the so-called Brauer tree algebras. They are Brauer graph algebras $A_{\mathcal{G}}$ given by a tree $\mathcal{G}$, that is, a graph without cycles, and there is at most one vertex in $V$ of multiplicity greater than $1$. It is known that Brauer tree algebras are the class of Brauer graph algebras of f\mbox inite-representation type \cite[Corollary 2.9]{S} and they appear as blocks of group algebras over f\mbox inite groups \cite[Theorem 4.12]{SY}. Theorem \ref{ubg} gives a characterization of Brauer tree algebras that are also UMP algebras, and they are the family of algebras in items $(3)(a)$ and $(3)(b)$.
\item Our classif\mbox ication for Brauer tree UMP algebras, coincides with the family of (symmetric monomial UMP) algebras appearing in the classif\mbox ication in \cite[Lemma 3.1 $(i)$]{Erd} for symmetric indecomposable of f\mbox inite-representation type algebras such that all indecomposable non--projective modules have $\Omega-$period two.
\end{enumerate}
\end{remark}

\subsection{Examples and discussion}\label{sse}
In this subsection, we use illustrative examples to analyze and challenge the conditions for formulating some of our main results and constructing tools. It is also worth mentioning that in \cite{CFR2}, the authors analyze another set of properties and consequences of this new theory and its tools, with interesting implications for the study of representation theory of monomial special multiserial algebras and some of their homological properties.

\begin{exam}\label{enonsmu} In this example, we show the importance of the property of $A$ of being special multiserial to apply Theorem \ref{thm:main}. Consider the following quiver.
$$Q: \xymatrix{\cdot\ar@<-0.5ex>[rr]_c & &\cdot\ar@<-0.5ex>[ll]_b\ar[rr]^d\ar@(ur,ul)[]_{a} & &\cdot}$$
and take $I=\langle a^3,bcb,ab,ca,cd\rangle$. Then, we have that $\omega_a=a$, $\omega_b=\omega_c=bc$, and $\omega_d=d$, and we obtain the following ramif\mbox ications graph.
$$G_{Q,I}: \ \xymatrix{\cdot_{\omega_a}\ar[rr] &  & \cdot_{\omega_d}\\ & \cdot_{\omega_b} & }$$
Hence, $\mathcal{D}_{Q,I}=\{L,N\}$ where $\xymatrix{L:\ \cdot_{\omega_a}\ar[rr] &  & \cdot_{\omega_d}}$ and $\xymatrix{N:\ \cdot_{\omega_b}}$. Therefore,
$$Q_L: \xymatrix{\cdot \ar[rr]^d\ar@(ur,ul)[]_{a} & & \cdot \ , & I_L=\langle a^3\rangle, & \mathcal{M}_L=\{a^2d\}},$$
$$Q_N: \xymatrix{\cdot\ar@<-0.5ex>[rr]_c & &\cdot\ar@<-0.5ex>[ll]_b\ , & I_N=\langle bcb\rangle, &\mathcal{M}_N=\{cbc\}}$$
Note that $A:=\Bbbk Q/I$ is locally monomial, but not special multiserial because $a^2$ and $ad$ are not in $I$. However, the weakly connected components of $G_{Q,I}$ are as in the f\mbox irst case of the form of the components of the ramif\mbox ications graphs of special multiserial algebras (see Lemma \ref{lem:compchar}). Then, we can def\mbox ine $\omega(N)$, $\mathcal{T}_N$, $\Omega_N$, $S_N$, and the $\omega$-relations in the same way as for special multiserial algebras. We have that $\omega(L)=\omega_a\omega_d=ad$ and $\omega(N)=\omega_b=bc$. Note that $\Omega_L=\emptyset$, $\mathcal{S}_L=\{a^3\}$, $\Omega_N=\emptyset$, and $\mathcal{S}_N=\{bcb\}$. Observe that for the $\omega$-relation $r=a^3$, the path $\omega(N(r))=ad$ is not cyclic, but $A$ is UMP. This implies that Theorem \ref{thm:main} can not be applied for locally monomial algebras in general.
\end{exam}

\begin{exam}\textbf{(The UMP property is not a Morita invariant)}
We provide an example which shows that the UMP property is not a Morita invariant. Indeed, consider the following quiver (\cite[Ex. 5.3, p.150]{Sch})
$$
\xymatrix{                                  
             & & 2 \ar[dr]^\beta &  \\
Q: & 1 \ar[ur]^\alpha \ar[rr]^\gamma &           & 4 \ar[dl]^\delta\\
       & & 3 \ar[ul]^\epsilon &}
$$
and def\mbox ine the two ideals $I_1=\langle\gamma\delta,\delta\epsilon\rangle$ and $I_2=\langle\gamma\delta-\alpha\beta\delta, \delta\epsilon\rangle$. Let $A_1=\Bbbk Q/I_1$ and $A_2=\Bbbk Q/I_2$ be the respective bound quiver algebras. It is easy to check that $A_1\cong A_2$ through the isomorphism given by $\alpha\mapsto\alpha$, $\beta\mapsto\beta$, $\gamma\mapsto\gamma-\alpha\beta$, $\delta\mapsto\delta$ and $\epsilon\mapsto\epsilon$. Nevertheless, $A_1$ is not a UMP algebra contrary to the case of $A_2$, which is a UMP algebra. In fact, the set of maximal paths are, respectively, $\mathcal{M}_1=\{\epsilon\gamma, \epsilon\alpha\beta\delta\}$ and $\mathcal{M}_2=\{\epsilon\alpha\beta\delta\}$, where in the former set the maximal paths coincides in the arrow $\epsilon$.

In this case, we have that $\omega_\alpha=\alpha\beta$, $\omega_\gamma=\gamma$ and $\omega_\delta=\delta\epsilon$. Likewise, the ramif\mbox ications graph $G_{Q,I_i}$, $i=1,2$, are def\mbox ined by
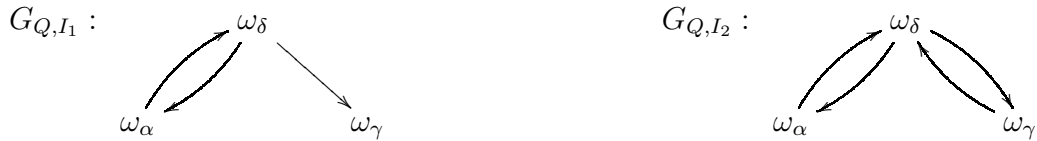
\begin{figure}[!htb]
   \begin{minipage}{0.46\textwidth}
\centering
$G_{Q,I_1}: \,\xymatrix{                               
              & \omega_{\delta} \ar@/^/[dl]\ar[dr] &  \\
\omega_{\alpha}\ar@/^/[ur]&           &\omega_{\gamma}}$
   \end{minipage}\hfill
   \begin{minipage}{0.52\textwidth}
 \centering
$G_{Q,I_2}:\,\xymatrix{                               
              &\omega_{\delta} \ar@/^/[dl]\ar@/^/[dr] &  \\
\omega_{\alpha} \ar@/^/[ur]&           &\omega_{\gamma} \ar@/^/[ul]}$
\end{minipage}
\end{figure}\\

In particular, the graphs $G_{Q,I_1}$ and $G_{Q,I_2}$ are not isomorphic.

\end{exam}

\section*{Acknowledgments}
The authors are especially grateful to professor Jos\'e A. V\'elez-Marulanda (VSU, GA, USA) for the very useful comments, suggestions and several enriching discussions about this work.

\section*{Declarations}

\subsection*{Funding}
The f\mbox irst and fourth authors were partially supported by CODI (Universidad de Antioquia, UdeA) by project numbers 2020-33305 and 2023-62291, respectively.

\subsection*{Ethical Approval}
Not applicable.

\subsection*{Availability of data and materials}
Not applicable.

\end{document}